\documentclass[a4paper]{amsart}

\usepackage[utf8]{inputenc}
\usepackage{amssymb}
\usepackage{graphicx}
\usepackage{hyperref}
\usepackage{xcolor}

\usepackage{hyperref}
\usepackage[capitalize,nameinlink,sort]{cleveref}

\usepackage{tikz}
\usetikzlibrary{arrows, automata, shapes, positioning}
\usetikzlibrary{arrows.meta}
\usetikzlibrary{decorations.pathmorphing,patterns,decorations.pathreplacing}
\usetikzlibrary{calc}
 \usepackage{float}

\newcommand{\infw}[1]{\mathbf{#1}}
\newcommand{\N}{\mathbb{N}}
\newcommand{\Z}{\mathbb{Z}}
\newcommand{\eps}{\varepsilon}
\DeclareMathOperator{\rep}{rep}

\DeclareMathOperator{\ab}{ab}
\DeclareMathOperator{\add}{add}
\def\modd#1 #2{#1\ \mbox{\rm (mod}\ #2\mbox{\rm )}}

\theoremstyle{plain}
\newtheorem{theorem}{Theorem}
\newtheorem{lemma}[theorem]{Lemma}
\newtheorem{corollary}[theorem]{Corollary}
\newtheorem{proposition}[theorem]{Proposition}
\newtheorem{conjecture}[theorem]{Conjecture}
\theoremstyle{definition}
\newtheorem{definition}[theorem]{Definition}
\newtheorem{remark}[theorem]{Remark}

\newtheorem{example}[theorem]{Example}
\newtheorem{problem}[theorem]{Problem}
\newtheorem{question}[theorem]{Question}

\numberwithin{equation}{section}

\usepackage{color}

\title{Additive word complexity and \texttt{Walnut}}

 \author{Pierre Popoli}
\address{Department of Mathematics, University of Li\`ege, Belgium}
\email{pierre.popoli@uliege.be}

\author{Jeffrey Shalliit}
\address{School of Computer Science, University of Waterloo, Canada}
\email{shallit@waterloo.ca}

 \author{Manon Stipulanti}
\address{Department of Mathematics, University of Li\`ege, Belgium}
\email{m.stipulanti@uliege.be}

\begin{document}

\begin{abstract}
In combinatorics on words, a classical topic of study is the number of specific patterns appearing in infinite sequences.
For instance, many works have been dedicated to studying the so-called factor complexity of infinite sequences, which gives the number of different factors (contiguous subblocks of their symbols), as well as abelian complexity, which counts factors up to a permutation of letters.
In this paper, we consider the relatively unexplored concept of additive complexity, which counts the number of factors up to additive equivalence.
We say that two words are additively equivalent if they have the same length and the total weight of their letters is equal.
Our contribution is to expand the general knowledge of additive complexity from a theoretical point of view and  consider various famous examples.
We show a particular case of an analog of the long-standing conjecture on the regularity of the abelian complexity of an automatic sequence.
In particular, we use the formalism of logic, and the software \verb|Walnut|, to decide related properties of automatic sequences.
We compare the behaviors of additive and abelian complexities, and we also consider the notion of abelian and additive powers.
Along the way, we present some open questions and conjectures for future work.
\end{abstract}

\maketitle

\bigskip
\hrule
\bigskip



\bigskip
\hrule
\bigskip

\section{Introduction}

Combinatorics on words is the study of finite and infinite sequences, also known as \emph{streams} or \emph{strings} in other theoretical contexts.
Although it is rooted in the work of Axel Thue, who was the first to study regularities in infinite words in the early 1900's, words became a systematic topic of combinatorial study in the second half of the 20th century~\cite{BerstelPerrin2007}.
Since then, many different approaches have been developed to analyze words from various points of view.
One of them is the celebrated \emph{factor} or \emph{subword complexity function}: given an infinite word $\infw{x}$ and a length $n\ge 0$, we compute the size of  $\mathcal{L}_n(\infw{x})$,  which contains all length-$n$ contiguous subblocks of $\infw{x}$,  also called \emph{factors} or \emph{subwords} in the literature.
One of the most famous theorems in combinatorics on words related to the factor complexity function is due to Morse and Hedlund in 1940~\cite{MorseHedlund1940}, where they obtained a characterization of ultimately periodic words.
As a consequence of this result, combinatorists defined binary aperiodic infinite words having the smallest possible factor complexity function, the so-called \emph{Sturmian words}.

Many other complexity functions have been defined on words depending on the properties combinatorists wanted to emphasize; see, for instance, the non-exhaustive list in the introduction of~\cite{ACSS2024+}.
As the literature on the topic is quite large, we only cite the so-called \emph{abelian complexity function}. Instead of counting all distinct factors, we count them up to abelian equivalence: two words $u$ and $v$ are \emph{abelian equivalent}, written $u \sim_{\ab} v$, if they are permutations of each other.
For instance, in English, \texttt{own}, \texttt{now}, and \texttt{won} are all abelian equivalent. 
For an infinite word $\infw{x}$, we let $\rho^{\ab}_{\infw{x}}$ denote its abelian complexity function.
In this paper, we study yet another equivalence relation on words; namely, {\it additive equivalence}.
Roughly, two words are additively equivalent if the total weight of their letters is equal.
 In the following,  for a word $w\in\Sigma^*$,  we let $|w|$ denote its \emph{length}, i.e.,  the number of letters it is composed of.  Furthermore,  for each letter $a\in\Sigma$, we let $|w|_a$ denote the number of $a$'s in $w$.

\begin{definition}
    Fix an integer $\ell\ge 1$ and the alphabet $\Sigma=\{0,1,\ldots,\ell\}$.
    Two words $u,v\in\Sigma^*$ are \emph{additively equivalent} if $|u|=|v|$ and $\sum_{i=0}^\ell i |u|_i = \sum_{i=0}^\ell i |v|_i$, which we write as $u\sim_{\add} v$.
\end{definition}

\begin{example}
    Over the three-letter alphabet $\{0,1,2\}$, we have $020\sim_{\add} 101$.
\end{example}

\begin{definition}
Fix an integer $\ell\ge 1$ and the alphabet $\Sigma=\{0,1,\ldots,\ell\}$. 
Let $\infw{x}$ be an infinite word on $\Sigma$.
    The \emph{additive complexity} of $\infw{x}$ is the function $\rho^{\add}_{\infw{x}} \colon \N \to \N, n\mapsto \#(\mathcal{L}_n(\infw{x})/{\sim_{\add}})$, i.e., length-$n$ factors of $\infw{x}$ are counted up to additive equivalence.
\end{definition}

Surprisingly, not so many results are known for additive equivalence and the corresponding complexity function, in contrast with the abundance of abelian results in combinatorics; see~\cite{RSZ-2010,RSZ-2011,Shallit-2021,Turek-2013,Turek-2015}, for example. Notice that over a two-letter alphabet, the concepts of abelian and additive complexity coincide.
Additive complexity was first introduced in~\cite{ABJS-2012}, where the main result states that bounded additive complexity implies that the underlying infinite word contains an additive $k$-power for every $k$ (an \emph{additive $k$-power} is a word $w$ that can be written as $x_1 x_2 \cdots x_k$ where the words $x_1,x_2,\ldots, x_k$ are all additively equivalent).


Later, words with bounded additive complexity were studied: first in~\cite{ABJS-2012}, and with more attention in~\cite{Banero-2013}.
In particular, equivalent properties of bounded additive complexity were found.
We also mention work on the particular case of constant additive and abelian complexities~\cite{Banero-2013,Currie-Rampersad-2011,RSZ-2011,Sahasrabudhe-2015}.
 As already observed, combinatorists expanded the notion of pattern avoidance to additive powers.
See the most recent preprint~\cite{AndradeMol-2024} for a nice exposition of the history.
For instance, Cassaigne et al.~\cite{CCSS2014} proved that the fixed point of the morphism $0\mapsto 03, 1\mapsto 43, 3\mapsto 1, 4\mapsto 01$ avoids additive cubes (see~\cref{sec:preliminaries} for concepts not defined in this introduction).
Rao~\cite{Rao2015} proved that it is possible to avoid additive cubes over a ternary alphabet and mentioned that the question about additive squares (in one dimension  and over the integers) is still open. 
Furthermore, we mention the following related papers.
Brown and Freedman \cite{Brown&Freedman:1987} also talked about the open problem about additive squares. 
{A notion of ``close'' additive squares is defined in~\cite{Brown-2012}.}
In~\cite{Halbeisen&Hungerbuhler:2000}, the authors showed that in every infinite word over a finite set of non-negative integers there is always a sequence of factors (not necessarily of the same length) having the same sum.
In~\cite{Pirillo&Varricchio:1994}, the more general setting of $k$-power modulo a morphism was studied.
Finally, in terms of computing the additive complexity of specific infinite words, to our knowledge, only that of a fixed point of a Thue--Morse-like morphism is known~\cite{CWW-2019}. 

In this paper, we expand our general knowledge of additive complexity functions of infinite words.
After giving some preliminaries in~\cref{sec:preliminaries}, we obtain several general results in~\cref{sec:general results}, and in~\cref{thm:effective procedure for addcompl} we prove a particular case of the conjecture below.
 Note that it is itself a particular case of the long-standing similar conjecture in an abelian context: the abelian complexity of a $k$-automatic sequence is a $k$-regular sequence~\cite{ParreauRigoRowlandVandomme-2015}.

\begin{conjecture}
\label{conj:ab and add kregular}
 The additive complexity of a $k$-automatic sequence is a $k$-regular sequence.
\end{conjecture}

In particular, the proof of our~\cref{thm:effective procedure for addcompl} relies on the logical approach to combinatorics on words: indeed, many properties of words can be phrased in first-order logic.
Based on this, Mousavi~\cite{Walnut1} designed the free software
\verb|Walnut| that allows one to automatically decide the truth of assertions about many properties for a large family of words.
See~\cite{Walnut2} for the formalism of the software and a survey of the combinatorial properties that can be decided.
In~\cref{sec:general results}, we show how \verb|Walnut| may be used in an ad-hoc way to give partial answers to~\cref{conj:ab and add kregular}.
In~\cref{sec:different behaviors}, we compare the behaviors of the additive and abelian complexity functions of various words.
We highlight the fact that they may behave quite differently, sometimes making use of \verb|Walnut|.
Motivated by the various behaviors we observe, we study in~\cref{sec:equality between add and ab} some words for which the additive and abelian complexity functions are in fact equal.
We end the paper by considering the related notions of abelian and additive powers.

\section{Preliminaries}
\label{sec:preliminaries}

For a general reference on words, we guide the reader to~\cite{Loth97}.
An \emph{alphabet} is a finite set of elements called \emph{letters}.
A \emph{word} over an alphabet $\Sigma$ is a finite or infinite sequence of letters from $\Sigma$. The \emph{length} of a finite word $w$, denoted $|w|$, is the number of letters it is made of.
The \emph{empty word} is the only $0$-length word, denoted by $\eps$.
For all $n\ge 0$, we let $\Sigma^n$ denote the 
set of all length-$n$ words over $\Sigma$.
We let $\Sigma^*$ denote the set of finite words over $A$, including the empty word, and equipped with the concatenation.
In this paper, we distinguish finite and infinite words by writing the latter in bold.
For each letter $a\in\Sigma$ and a word $w\in\Sigma^*$, we let $|w|_a$ denote the number of $a$'s in $w$.
Let us assume that the alphabet $\Sigma=\{a_1<\cdots<a_k\}$ is ordered.
For a word $w\in\Sigma^*$, we let $\Psi(w)$ denote the \emph{abelianization} or \emph{Parikh vector} $(|w|_{a_1},\ldots,|w|_{a_k})$, which counts the number of different letters appearing in $w$.   For example, over the alphabet $\{\texttt{e} < \texttt{l} <  \texttt{s} < \texttt{v}\}$, we have $\Psi({\tt sleeveless}) = (4,2,3,1)$.

A \emph{factor} of a word is one of its (contiguous) subblocks.
For a given word $\infw{x}$, for all $n\ge 0$, we let $\mathcal{L}_n(\infw{x})$ denote the set of length-$n$ factors of $\infw{x}$.
A \emph{prefix} (resp., \emph{suffix}) is a starting (resp., ending) factor.
A prefix or a suffix is \emph{proper} if it is not equal to the initial word.
Infinite words are indexed starting at $0$. For such a word $\infw{x}$, we let $\infw{x}(n)$ denote its $n$th letter with $n\ge 0$ and, for $0\le m\le n$, we let $\infw{x}[m..n]$ denote the factor $\infw{x}(m)\cdots\infw{x}(n)$.

Let $\Sigma$ and $\Gamma$ be finite alphabets.
A \emph{morphism} $f \colon \Sigma^* \to \Gamma^*$ is a map satisfying $f(uv)=f(u)f(v)$ for all $u,v\in \Sigma^*$.
In particular, $f(\eps)=\eps$, and $f$ is entirely determined by the images of the letters in $\Sigma$.
For an integer $k\ge 1$, a morphism is \emph{$k$-uniform} if it maps each letter to a length-$k$ word.
A $1$-uniform morphism is called a \emph{coding}.
A sequence $\infw{x}$ is \emph{morphic} if there exist a morphism $f\colon \Sigma^* \to \Sigma^*$, a coding $g\colon \Sigma^* \to \Gamma^*$, and a letter $a \in \Sigma$ such that $\infw{x} = g(f^{\omega}(a))$, where $f^{\omega}(a) = \lim_{n\to \infty} f^n(a)$.
The latter word $f^{\omega}(a)$ is a \emph{fixed point} of $f$.

Introduced by Cobham~\cite{Cobham72} in the early 1970s, automatic words have several equivalent definitions depending on the point of view one wants to adopt.
For the case of integer base numeration systems, a comprehensive presentation of automatic sequences is~\cite{AS03}, while~\cite{RigoMaes2002,Shallit1988} treat the case of more exotic numeration systems.
We start with the definition of positional numeration systems.
Let $U=(U(n))_{n\ge 0}$ be an increasing sequence of integers with $U(0)=1$.
A positive integer $n$ can be decomposed, not necessarily uniquely, as $n=\sum_{i=0}^t c(i)\, U(i)$ with non-negative integer coefficients $c(i)$.
If these coefficients are computed greedily, then for all $j< t$ we have $\sum_{i=0}^j c(i)\, U(i)<U(j+1)$ and $\rep_U(n)=c(t)\cdots c(0)$ is said to be the {\em (greedy) $U$-representation} of $n$. 
By convention, that of $0$ is the empty word~$\varepsilon$,  and the greedy representation of $n>0$ starts with a non-zero digit.
A sequence $U$ satisfying all the above conditions defines a \emph{positional numeration system}.
Let $U=(U(n))_{n\ge 0}$ be such a numeration system.
A sequence $\mathbf{x}$ is {\em $U$-automatic} if there exists a deterministic finite automaton with output (DFAO) $\mathcal{A}$ such that, for all $n\ge 0$, the $n$th term $\infw{x}(n)$ of $\infw{x}$ is given by the output $\mathcal{A}(\rep_U(n))$ of $\mathcal{A}$.
In the particular case where $U$ is built on powers of an integer $k\ge 2$, then $\mathbf{x}$ is said to be {\em $k$-automatic}.
It is known that a sequence is $k$-automatic if and only if it is the image, under a coding, of a fixed point of a $k$-uniform morphism~\cite{AS03}.

A generalization of automatic sequences to infinite alphabets is the notion of regular sequences~\cite{AS03,RigoMaes2002,Shallit1988}.
Given a positional numeration system $U=(U(n))_{n\ge 0}$, a sequence $\infw{x}$ is \emph{$U$-regular} if there exist a column vector $\lambda$, a row vector $\gamma$ and \emph{matrix-valued} morphism $\mu$, i.e.,  the image of each letter is a matrix, such that $\infw{x}(n)=\lambda \mu(\rep_U(n)) \gamma$.
Such a system of matrices forms a \emph{linear representation} of $\infw{x}$.
In the particular case where $U$ is built on powers of an integer $k\ge 2$, then $\mathbf{x}$ is said to be {\em $k$-regular}.
Another definition of $k$-regular sequences is the following one~\cite{AS03}.
Consider a sequence $\infw{x}$ and an integer $k\ge 2$.
The \emph{$k$-kernel} of $\infw{x}$ is the set of subsequences of the form $(\infw{x}(k^e n + r))_{n\ge 0}$ where $e\geq 0$ and $r\in\{0,1,\ldots,k^e-1\}$.
Equivalently, a sequence is $k$-regular if the $\Z$-module generated by its $k$-kernel is finitely generated.
A sequence is then $k$-automatic if and only if its $k$-kernel is finite~\cite{AS03}.

Introduced in 2001 by Carpi and Maggi~\cite{CarpiMaggi2001}, synchronized sequences form a family between automatic and regular sequences.
Given a positional numeration system $U=(U(n))_{n\ge 0}$, a sequence $\infw{x}$ is \emph{$U$-synchronized} if there exists a deterministic finite automaton (DFA) that recognizes the language of $U$-representations of $n$ and $\infw{x}(n)$ in parallel. 

\section{General results}
\label{sec:general results}

In this section, we gather general results on the additive complexity of infinite words. Since abelian equivalence implies additive equivalence, we have the following lemma.

\begin{lemma}
\label{lem:comparison-abb-ad}
    For all infinite words $\infw{x}$, we have $\rho^{\add}_{\infw{x}}(n) \le \rho^{\ab}_{\infw{x}}(n)$ for all $n\ge 0$.
\end{lemma}

As in the case of abelian complexity, we have the following lower and upper bounds for additive complexity.
See~\cite[Rk.~4.07]{Coven-Hedlund-1973} and~\cite[Thm~2.4]{RSZ-2011}.

\begin{lemma}
\label{lem:bounds on additive complexity}
    Let $k\ge 1$ be an integer and let $\infw{x}$ be an infinite word on $\{a_1<\cdots<a_k\}$.
    We have $1 \le \rho^{\add}_{\infw{x}}(n) \le \binom{n+k-1}{k-1}$ for all $n\ge 0$.
\end{lemma}

Note that the lower bound of the previous result is reached for (purely) periodic sequences.
The story about the upper bound is a little more puzzling.
In fact, for a window length $N\ge 1$, we can find an alphabet and a sequence over this alphabet for which its additive complexity reaches the stated upper bound on its first $N$ values.
Indeed, fix an integer $k\ge 3$ and an alphabet $\Sigma=\{a_1<\cdots<a_k\}$ of integers.
Consider the Champernowne-like sequence defined on $\Sigma$ by concatenating all words of $\Sigma^*$ in lexicographic order.
Then, for all $N\ge 1$, 
we can find a valuation of $\Sigma$ (i.e., a distribution of integral values for the letters of $\Sigma$) such that $\rho^{\add}_{\infw{x}}(n) = \binom{n+k-1}{k-1}$ for all $n\le N$.
However, it does not seem possible to find a sequence for which its additive complexity always reaches the upper bound.
This already highlights the unusual fact that the underlying alphabet of the words plays a crucial role in additive complexity.

The classical theorem of Morse and Hedlund~\cite{MorseHedlund1940} characterizes ultimately periodic infinite words by means of their factor complexity.
With the notion of additive complexity, we no longer have a characterization, only the implication below.
The converse of~\cref{pro:ult pero implies add bounded} does not hold, as illustrated by several examples in~\cref{sec:different behaviors}.

\begin{proposition}
\label{pro:ult pero implies add bounded}
The additive complexity of an ultimately periodic word is bounded.
\end{proposition}

Similarly, balanced words may be characterized through their abelian complexity. A word $\infw{x}$ is said to be $C$-\emph{balanced} if $| |u|_a-|v|_a | \leq C$ for all $a\in \Sigma$ and all factors $u,v$ of $\infw{x}$  of equal length. Richomme, Saari and Zamboni~\cite[Lemma~3]{RSZ-2011} proved that an infinite word $\infw{x}$ is $C$-balanced for some $C\geq 1$ if and only if $\rho^{\ab}_{\infw{x}}$ is bounded.
In our case, we only have one implication, as stated in~\cref{pro:C-balanced implies bounded}, and we also provide an upper bound.

\begin{proposition}
\label{pro:C-balanced implies bounded}
    Let $\Sigma=\{a_1<\cdots<a_k\}$ and let $\infw{x}$ be a $C$-balanced word on $\Sigma$. Then the additive complexity of $\infw{x}$ is bounded by a constant. More precisely,  we have $\rho^{\add}_{\infw{x}}(n)
 \leq 
 C \sum_{i=1}^{\lceil k/2\rceil} (a_i - a_{k+1-i}) + 1$
 for all $n\geq 0$.
\end{proposition}

\begin{proof}
    For all length-$n$ factors $y,z$ of $\infw{x}$ and $a\in\Sigma$, we have $||y|_a-|z|_a|\leq C$. 
    So the largest possible gap between the sum of letters of $y$ and the sum of letters of $z$ is when,  for
     all $i\in\{1,\ldots,\lceil k/2\rceil\}$, $|y|_{a_i}=|z|_{a_i}+C$ and $|y|_{a_{k+1-i}}=|z|_{a_{k+1-i}}-C$,  or vice versa (in short,  we swap $C$ letters from $a_k$ to $a_1$, $C$ others from $a_{k-1}$ to $a_2$,  and so on and so forth).
\end{proof}

Note that \cref{pro:C-balanced implies bounded} is a particular case of~\cite[Theorem 4]{Banero-2013}.
However, there are infinite words with bounded additive complexity and unbounded abelian complexity, making them not balanced.
For an example, see~\cref{sec:bounded add and unbounded ab}.

Computing additive and abelian complexity might be ``easy'' in some cases.
Recently, Shallit~\cite{Shallit-2021} provided a general method to compute the abelian complexity of an automatic sequence under some hypotheses. 

\begin{theorem}[{\cite[Thm.~1]{Shallit-2021}}]
    Let $\infw{x}$ be a sequence that is automatic in some regular numeration system. Assume that
    \begin{enumerate}
        \item the abelian complexity $\rho^{\ab}_{\infw{x}}$ of $\infw{x}$ is bounded above by a constant, and
        \item the Parikh vectors of length-$n$ prefixes of $\infw{x}$ form a synchronized sequence.
    \end{enumerate}
    Then $\rho^{\ab}_{\infw{x}}$ is an automatic sequence and the DFAO computing it is effectively computable. 
\end{theorem}

We obtain an adapted version in the framework of additive complexity.

\begin{theorem}
\label{thm:effective procedure for addcompl}
    Let $\infw{x}$ be a sequence that is automatic in some additive numeration system.
    Assume that
    \begin{enumerate}
        \item \label{it:cond-1} the additive complexity $\rho^{\add}_{\infw{x}}$ of $\infw{x}$ is bounded above by a constant, and
        \item \label{it:cond-2} the Parikh vectors of length-$n$ prefixes of $\infw{x}$ form a synchronized sequence. 
    \end{enumerate}
    Then $\rho^{\add}_{\infw{x}}$ is an automatic sequence and the DFAO computing it is effectively computable.
\end{theorem}

\begin{proof}
Let $\Sigma=\{a_1<\cdots<a_k\} \subset \mathbb{N}$ be an ordered finite alphabet. The \emph{weighted Parikh vector} of a finite word $w\in \Sigma^{*}$ is $\psi^{*}(w)=(a_1 |w|_{a_1}, \ldots, a_k |w|_{a_k})$. Then two words $x,y$ are additively equivalent if and only if $\sum_{a\in \Sigma}[\psi^{*}(x)]_a=\sum_{a\in \Sigma}[\psi^{*}(y)]_a$, where $[\psi^{*}(x)]_a$ designates the $a$th component of the vector $\psi^{*}(x)$.
    We adapt the proof of {\cite[Thm.~1]{Shallit-2021}} in the framework of the additive complexity. The steps to find the automaton computing the additive complexity $\rho^{\add}_{\infw{x}}$ are the following: \begin{enumerate}
        \item Since the Parikh vectors of length-$n$ prefixes of $\infw{x}$ form a synchronized sequence by assumption, so are the weighted Parikh vectors for arbitrary length-$n$ factors $\infw{x}[i..i+n-1]$. This is expressible in first-order logic.
        \item For $i\geq 0$ and $n\geq 1$, let us denote $\Delta_{\infw{x}}(i,n)$ the following integer \[\Delta_{\infw{x}}(i,n)=\sum_{a\in \Sigma}[\psi^{*}(\infw{x}[i..i+n-1])]_a-\sum_{a\in \Sigma}[\psi^{*}(\infw{x}[0..n-1])]_a.\] The additive complexity $\rho^{\add}_{\infw{x}}$ is bounded if and only if there is a constant $C$ such that the cardinality of the set $A^{*}_n:= \lbrace \Delta_{\infw{x}}(i,n) : i\geq 0  \rbrace$, is bounded above by $C$ for all $n\geq 1$. 
        \item In this case, the range of possible values of $A^{*}_n$ is finite (it may take at most $2C+1$ values) and can be computed algorithmically. 
        \item Once this range is known, there are finitely many possibilities for $\Delta_{\infw{x}}(i,n)$ for all $i\geq 0$. Then, we compute the set $S$ of all of these possibilities. 
        \item Once we have $S$, we can test each of the finitely many values to see if it occurs for some $n$, and we obtain an automaton recognizing those $n$ for which it does. 
        \item All the different automata can then be combined into a single DFAO computing $\rho^{\add}_{\infw{x}}(n)$, using the direct product construction. 
    \end{enumerate}
    This finishes the proof.
\end{proof}

\begin{remark}
\label{rk: constructive proof vs semi-group trick halts}
The advantage of the proof above is that it is constructive.
However, in practice, it will be more convenient to use the so-called \emph{semigroup trick} algorithm, as discussed in~\cite[\S~4.11]{Walnut2}. This algorithm should be used when a regular sequence is believed to be automatic, i.e., when it takes only finitely many values. The semigroup trick algorithm halts if and only if the sequence is automatic and produces a DFAO if this is the case. Therefore,~\cref{thm:effective procedure for addcompl} ensures that, under some mild hypotheses, the algorithm halts. 


\end{remark}

\cref{thm:effective procedure for addcompl} may be applied to a particular family of infinite words: those that are generated by so-called Parikh-collinear morphisms.
In recent years, combinatorists have been studying them; see, e.g., \cite{CRSZ-2011,Allouche-Dekking-Queffelec-2021,RSW-2023,RSW-2024}.

\begin{definition}
A morphism $\varphi \colon \Sigma^* \to \Delta^*$ is \emph{Parikh-collinear} if the Parikh vectors $\Psi(\varphi(a))$, $a\in \Sigma$, are collinear (or pairwise $\Z$-linearly dependent).
In other words, the associated \emph{adjacency matrix} of $\varphi$, i.e.,  the matrix whose columns are the vectors $\Psi(\varphi(a))$, for all $a\in \Sigma$, has rank $1$.
\end{definition}

\begin{theorem}[{\cite{RSW-2023,RSW-2024}}]
\label{thm:Parikh-collinear-automatic-ab}
Let $\varphi \colon \Sigma^* \to \Sigma^*$ be a Parikh-collinear morphism prolongable on the letter $a$, and write $\infw{x}:=\varphi^{\omega}(a)$.
Then the abelian complexity function $\rho^{\ab}_{\infw{x}}$ of $\infw{x}$ is $k$-automatic for $k = \sum_{b \in \Sigma}|\varphi(b)|_b$.
Moreover, the automaton generating $\rho^{\ab}_{\infw{x}}$ can be effectively computed given $\varphi$ and $a$.
\end{theorem}

Putting together~\cref{lem:comparison-abb-ad,thm:Parikh-collinear-automatic-ab}, we obtain the following.

\begin{corollary}
\label{cor:bounded add}
    Let $\infw{x}$ be a fixed point of a Parikh-collinear morphism.
    Then the abelian and additive complexity functions of $\infw{x}$ are bounded.
\end{corollary}

\begin{theorem}
\label{pro:Parikh collinear automatic additive complexity}
    Let $\varphi \colon \Sigma^* \to \Sigma^*$ be a Parikh-collinear morphism prolongable on the letter $a$, and write $\infw{x}:=\varphi^{\omega}(a)$.
The additive complexity function $\rho^{\add}_{\infw{x}}$ of $\infw{x}$ is $k$-automatic for $k = \sum_{b \in \Sigma}|\varphi(b)|_b$.
Moreover, the automaton generating $\rho^{\add}_{\infw{x}}$ can be effectively computed given $\varphi$ and $a$.
\end{theorem}

\begin{proof}
    By~\cref{cor:bounded add}, $\rho^{\add}_{\infw{x}}$ is bounded by a constant, so~\cref{it:cond-1} of~\cref{thm:effective procedure for addcompl} is satisfied.
    Then~\cref{it:cond-2} of of~\cref{thm:effective procedure for addcompl} holds by~\cite[Lemma~26]{RSW-2024}.
    Hence,~\cref{thm:effective procedure for addcompl} allows to finish the proof.   
\end{proof}

We now give a detailed example of~\cref{pro:Parikh collinear automatic additive complexity}.
Let $f \colon \{0,1,2\}^* \to \{0,1,2\}$ be defined by $0\mapsto 012, 1\mapsto 112002, 2\mapsto \varepsilon$. Since the three vectors $\Psi(f(0))=(1,1,1)$, $\Psi(f(1))=(2,2,2)$ and $\Psi(f(2))=(0,0,0)$ are collinear, it follows that $f$ is Parikh-collinear.
Consider $\infw{x}=012112002112002\cdots$, the fixed point of $f$ starting with $0$. In \cite{RSW-2024}, the authors proved that the abelian complexity of $\infw{x}$ is equal to the eventually periodic word $135(377)^{\omega}$. We have a similar result for additive complexity. 

\begin{proposition}\label{prop:example_Parikh_collinear}
  Let $f \colon \{0,1,2\}^* \to \{0,1,2\}, 0\mapsto 012, 1\mapsto 112002, 2\mapsto \varepsilon$.
The additive complexity of the fixed point $\infw{x}=0121120022112002\cdots$ of $f$ is equal to $134(355)^{\omega}$. 
  \end{proposition}

\begin{proof}
Computing $\sum_{a=0}^2 \lvert f(a) \rvert_a=3$, we know from classical results that $\infw{x}$ is $3$-automatic.
We thus know that $\infw{x}$ is generated by a $3$-uniform morphism.
Following the procedure of \cite{RSW-2023}, we have $\infw{x}=\tau(h^{\omega}(0))$ with $h:0,6\mapsto 012$, $1,4\mapsto 134$, $2,3,5\mapsto 506$, and the coding $\tau:0,5\mapsto 0$,  $1,3\mapsto 1$, and $2,4,6\mapsto 2$. 

In \verb|Walnut|, we can compute the synchronized functions \texttt{fac0}, \texttt{fac1} and \texttt{fac2} that computes the number of letter $0$, $1$ and $2$ in every factor of $\infw{x}$, see~\cite{RSW-2024} for more details.
First, we introduce the $3$-automatic word $\infw{x}$ to \verb|Walnut| as follows:
\begin{verbatim}
morphism h "0->012 1->134 2->506 3->506 4->134 5->506 6->012";
morphism tau "0->0 1->1 2->2 3->1 4->2 5->0 6->2";
promote H h;
image X tau H;
\end{verbatim}
From~\cite{RSW-2023}, we know that the sequence mapping $n\ge 0$ to the number of each letter in the prefix of length $n+1$ of $\infw{x}$ is synchronized. In \verb|Walnut|, we require the following commands: 
\begin{verbatim}
def cut "?msd_3 n=0 | (n>=3 & X[n-1]=@2 & ~(X[n-3]=@1))";
def prev "?msd_3 x<=n & $cut(x) & (Ay (y>x & y<=n)=>~$cut(y))";
def prefn0 "?msd_3 (n<=2 & y=1) | (3<=n & Em,z ($prev(n,m) 
    & 3*y=m+3*z & ((X[m]=@0 & z=1) | (X[m]=@1 & ((n<m+3 & z=0) 
    | (n=m+3 & z=1) | (n>=m+4 & z=2))))))";
def prefn1 "?msd_3 Em,z $prev(n,m) & 3*y=m+3*z 
    & ((X[m]=@0 & ((m=n & z=0) | (n>=m+1 & z=1))) 
    | (X[m]=@1 & ((m=n & z=1) | (n>=m+1 & z=2))))";
def prefn2 "?msd_3 Em,z $prev(n,m) & 3*y=m+3*z 
    & ((X[m]=@0 & ((n<m+2 & z=0) | (m+2=n & z=1))) 
    | (X[m]=@1 & ((n<m+2 & z=0) | (n>=m+2 & n<m+5 & z=1) 
    | (n=m+5 & z=2))))";
\end{verbatim} The next step is to determine the number of each letter in every factor of $\infw{x}$. We compute the corresponding synchronized functions $n\mapsto |\infw{x}[i..i+n-1]|_a$ for $a\in\{0,1,2\}$ in the following way: \begin{verbatim}
def fac0 "?msd_3 Aq,r ($prefn0(i+n,q) & $prefn0(i,r)) => (q=r+s)":
def fac1 "?msd_3 Aq,r ($prefn1(i+n,q) & $prefn1(i,r)) => (q=r+s)":
def fac2 "?msd_3 Aq,r ($prefn2(i+n,q) & $prefn2(i,r)) => (q=r+s)":
\end{verbatim} 

Next, wecompute additive complexity of the fixed point $\infw{x}$ of $f$. So we test whether the factors $u=\infw{x}[i..i+n-1]$ and $v=\infw{x}[j..j+n-1]$ of $\infw{x}$ are additively equivalent.
For that, it is enough to check the equality between the quantities $|u|_1+2|u|_2$ and $|v|_1+2|v|_2$.

\begin{verbatim}
def addFacEq "?msd_3 Ep,q,r,s $fac1(i,n,p) & $fac2(i,n,q) 
    & $fac1(j,n,r) & $fac2(j,n,s) & p+2*q=r+2*s":
\end{verbatim}
Finally, we write that $\infw{x}[i..i+n-1]$ is a novel occurrence of a length-$n$ factor of $\infw{x}$ representing its additive equivalence class and obtain a linear representation for the number of such positions $i$ as follows:

\begin{verbatim}
eval addCompRepLin n "?msd_3 Aj j<i => ~$addFacEq(i,j,n)":
\end{verbatim} \verb|Walnut| then returns a linear representation of size $55$.

The first step is to take the linear representation computed by \verb|Walnut|, and minimize it. The result is a linear representation of rank $7$, using the algorithm in \cite[\S~2.3]{Berstel&Reutenauer:2011}. 
Once we have this linear representation, we can carry out the so-called semigroup trick algorithm, as discussed in~\cite[\S~4.11]{Walnut2}.
As it terminates, we prove that the additive complexity of the word $\infw{x}$ is bounded, and takes on only the values $\{1,3,4,5\}$ for $n \geq 0$. Furthermore, it produces a $4$-state DFAO computing the additive complexity, called \verb|addCompExample|, that we display in~\Cref{fig:automaton-example}. 

\begin{figure}[ht] 
\begin{center}
\begin{tikzpicture}
\tikzstyle{every node}=[shape=circle,fill=none,draw=black,minimum size=30pt,inner sep=2pt]
\node(a) at (0,0) {$q_0/1$};
\node(b) at (2.5,0) {$q_1/3$};
\node(c) at (0,-2.5) {$q_2/4$};
\node(d) at (2.5,-2.5) {$q_3/5$};

\tikzstyle{every node}=[shape=circle,fill=none,draw=none,minimum size=10pt,inner sep=2pt]
\node(a0) at (-1,0) {};

\tikzstyle{every path}=[color =black, line width = 0.5 pt]
\tikzstyle{every node}=[shape=circle,minimum size=5pt,inner sep=2pt]
\draw [->] (a0) to [] node [] {}  (a);

\draw [->] (a) to [loop above] node [] {$0$}  (a);
\draw [->] (a) to [] node [above] {$1$}  (b);
\draw [->] (a) to [] node [left] {$2$}  (c);

\draw [->] (b) to [loop above] node [] {$0$}  (b);
\draw [->] (b) to [bend left] node [right] {$1,2$}  (d);

\draw [->] (c) to [] node [left] {$0$}  (b);
\draw [->] (c) to [] node [below] {$1,2$}  (d);

\draw [->] (d) to [loop right] node [] {$1,2$}  (d);
\draw [->] (d) to [bend left] node [left] {$0$}  (b);

;
\end{tikzpicture}
\caption{A four-state DFAO computing the additive complexity of the fixed point of $f \colon \{0,1,2\}^* \to \{0,1,2\}, 0\mapsto 012, 1\mapsto 112002, 2\mapsto \varepsilon$.}
\label{fig:automaton-example}
\end{center}
\end{figure}
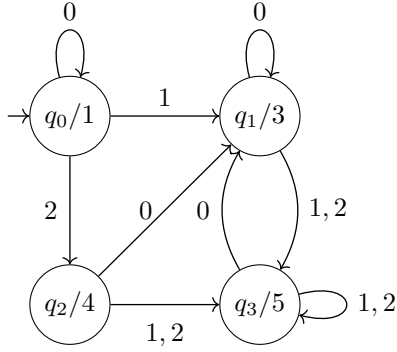

By inspecting this DFAO, we easily prove that the additive complexity of $\infw{x}$ is $134(355)^{\omega}$. This could also be checked easily with \verb|Walnut| with the following commands: 
\begin{verbatim} 
reg form3 msd_3 "0*(0|1|2)*0":
eval check3 "?msd_3 An ($form3(n) & n>=3) => addCompExample[n]=@3":
reg form5 msd_3 "0*(0|1|2)*(1|2)":
eval check5 "?msd_3 An ($form5(n) & n>=3) => addCompExample[n]=@5":
\end{verbatim} and both return \verb|True|.  Notice that these two forms cover all integers $n\geq 3$ and the first few values can be checked by hand.
 \end{proof}

\section{Different behaviors and curiosities}
\label{sec:different behaviors}

In this section, we exhibit different behaviors between the additive and abelian complexity functions by making use of the software \verb|Walnut|.
By~\cref{lem:comparison-abb-ad}, the behavior of additive complexity of a sequence is constrained by its abelian complexity.
Here we show that the functions may behave differently; in particular, see~\cref{sec:bounded add and unbounded ab}.

\subsection{Bounded additive and abelian complexities}

\subsubsection{The Tribonacci word}
\label{sec: Tribonacci}

The \emph{Tribonacci word} $\infw{tr}$ is the fixed point of the morphism $0\mapsto 01$, $1\mapsto 02$, $2\mapsto 0$. This well-known word belongs to the family of episturmian words, a generalization of the famous Sturmian words. This word is \emph{Tribonacci}-automatic, where the underlying numeration system is built on the sequence of \emph{Tribonacci numbers} defined by $T(0)=1$, $T(1)=2$, $T(2)=4$, and $T(n)=T(n-1)+T(n-2)+T(n-3)$ for all $n\ge 3$. Notice that this word is not the fixed point of a Parikh-collinear morphism; otherwise it would be $k$-automatic for some integer $k\geq 2$. A generalization of Cobham's theorem for substitutions~\cite{Durand2011} would then imply that  $\infw{tr}$ is ultimately periodic. The possible values of the abelian complexity of the word $\infw{tr}$ were studied in~\cite[Thm.~1.4]{RSZ-2010}.
Also see~\cref{fig:add-complexity-Trib}.

\begin{theorem}[{\cite[Thm.~1.4]{RSZ-2010}}]
Let $\infw{tr}$ be the Tribonacci word, i.e., the fixed point of the morphism $0\mapsto 01$, $1\mapsto 02$, $2\mapsto 0$.
The abelian complexity function $\rho^{\ab}_{\infw{tr}}$ takes on only the values in the set $\{3, 4, 5, 6, 7\}$ for $n\ge 1$.    
\end{theorem}

This result was reproved by Shallit~\cite{Shallit-2021} using \verb|Walnut| by providing an automaton computing $\rho^{\ab}_{\infw{tr}}$. Furthermore, this automaton allows us to prove that each value is taken infinitely often. We prove the following result concerning the additive complexity of the Tribonacci word.
See again~\cref{fig:add-complexity-Trib}.

\begin{theorem} \label{thm:tribo}
Let $\infw{tr}$ be the Tribonacci word, i.e., the fixed point of the morphism $0\mapsto 01$, $1\mapsto 02$, $2\mapsto 0$.
The additive complexity function $\rho^{\add}_{\infw{tr}}$ takes on only the values in the set $\{3, 4, 5\}$ for $n\ge 1$. Furthermore, each of the three values is taken infinitely often and it is computed by a 76-state Tribonacci DFAO.
\end{theorem}

\begin{proof}

We reuse some ideas (especially, \verb|Walnut| code) from~\cite{Shallit-2021,Shallit-2023}. The Tribonacci word is stored as \texttt{TRL} in \verb|Walnut|. The synchronized function \texttt{rst} takes the Tribonacci representations of $m$ and $n$ in parallel and accepts if $(n)_T$ is the right shift of $(m)_T$. In \verb|Walnut|, the following three predicates allow us to obtain DFAO's that compute the maps $n\mapsto |\infw{tr}[0..n-1]|_a$ for $a\in\{0,1,2\}$, i.e., the number of letters $0,1,2$ in the length-$n$ prefix of the Tribonacci word $\infw{tr}$. 
Note that the predicates are obtained using a special property of $\infw{tr}$; for a full explanation,  see~\cite[Sec.~3]{Shallit-2021}.

\begin{verbatim}
def tribsync0 "?msd_trib Ea Eb (s=a+b) & ((TRL[n]=@0)=>b=0) 
    & ((TRL[n]=@1)=>b=1) & $rst(n,a)":
def tribsync1 "?msd_trib Ea Eb Ec (s=b+c) & ((TRL[a]=@0)=>c=0) 
    & ((TRL[a]=@1)=>c=1) & $rst(n,a) & $rst(a,b)":
def tribsync2 "?msd_trib Ea Eb Ec Ed (s=c+d) & ((TRL[b]=@0)=>d=0) 
    & ((TRL[b]=@1)=>d=1) & $rst(n,a) & $rst(a,b) & $rst(b,c)":
\end{verbatim} From now on, we follow the same steps as \cref{prop:example_Parikh_collinear}. First, we compute the Tribonacci synchronized functions $n\mapsto |\infw{tr}[i..i+n-1]|_a$ for $a\in\{0,1,2\}$, that are

\begin{verbatim}
def tribFac0 "?msd_trib Aq Ar ($tribsync0(i+n,q) 
    & $tribsync0(i,r)) => (q=r+s)":
def tribFac1 "?msd_trib Aq Ar ($tribsync1(i+n,q) 
    & $tribsync1(i,r)) => (q=r+s)":
def tribFac2 "?msd_trib Aq Ar ($tribsync2(i+n,q) 
    & $tribsync2(i,r)) => (q=r+s)":
\end{verbatim} Next, we compute the additive equivalence between two factors, that is the following Tribonacci synchronized function

\begin{verbatim}
def tribAddFacEq "?msd_trib Ep,q,r,s $tribFac1(i,n,p) & $tribFac2(i,n,q) 
    & $tribFac1(j,n,r) & $tribFac2(j,n,s) & p+2*q=r+2*s":
\end{verbatim} Finally, we obtain a linear representation, as defined at the end of \cref{sec:preliminaries}, of the additive complexity as follows

\begin{verbatim}
eval tribAddCompRepLin n "?msd_trib Aj j<i => ~$tribAddFacEq(i,j,n)":
\end{verbatim} And \verb|Walnut| then returns a linear representation of size $184$. Then we apply the same procedure than in~\cref{prop:example_Parikh_collinear}.

After minimization, the result is a linear representation of rank $62$ and we carry out the semigroup trick.   This algorithm terminates, which proves that the additive complexity of the Tribonacci word is bounded, and takes on only the values $\{1,3,4,5 \}$ for $n \geq 0$. Furthermore, it produces a $76$-state DFAO computing the additive complexity.
In \verb|Walnut|, let us import this DFAO under the name \verb|TAC|.
To show that each value appears infinitely often, we test the following three predicates
\begin{verbatim}
eval tribAddComp_3 "?msd_trib An Em (m>n) & TAC[m]=@3":
eval tribAddComp_4 "?msd_trib An Em (m>n) & TAC[m]=@4":
eval tribAddComp_5 "?msd_trib An Em (m>n) & TAC[m]=@5":
\end{verbatim}
and \verb|Walnut| then returns \verb|TRUE| each time.   
\end{proof}

\begin{figure}[h]
    \centering
    \includegraphics[scale=0.45]{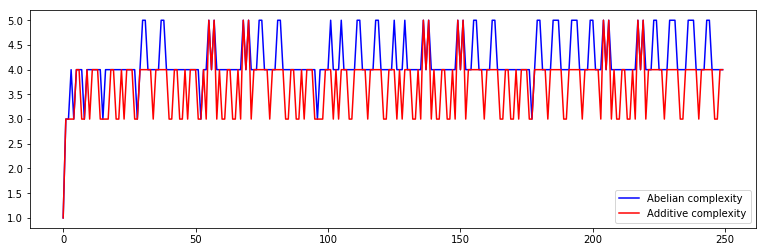}
    \caption{The first few values of the abelian and additive complexities for the Tribonacci word.}
    \label{fig:add-complexity-Trib}
\end{figure}

\begin{remark}
    From the automaton, which is too large to display here, it is easy to find infinite families for each value of the additive complexity function. Indeed, it suffices to detect a loop in the automaton leading to a final state for each value. For instance, we have the following infinite families: 
    \begin{itemize}
        \item[(a)] If $(n)_T=100(100)^k$, for $k\geq 0$, then $\rho^{\add}_{\infw{tr}}(n)=3$. 
        \item[(b)] If $(n)_T=1101(01)^k$, for $k\geq 0$, then $\rho^{\add}_{\infw{tr}}(n)=4$. 
        \item[(c)] If $(n)_T=1101001100(1100)^k$, for $k\geq 0$, then $\rho^{\add}_{\infw{tr}}(n)=5$. 
    \end{itemize}
    One can check with \verb|Walnut| that these infinite families are convenient with the following commands \begin{verbatim}
reg form3 msd_trib "0*100(100)*":
reg form4 msd_trib "0*1101(01)*":
reg form5 msd_trib "0*1101001100(1100)*":
eval check3 "?msd_trib An ($form3(n) & n>=1) => TAC[n]=@3":
eval check4 "?msd_trib An ($form4(n) & n>=1) => TAC[n]=@4":
eval check5 "?msd_trib An ($form5(n) & n>=1) => TAC[n]=@5":
\end{verbatim} which returns \verb|TRUE| for each command. One can also notice that from the automaton, we can build infinitely many infinite families of solutions of each of those values. However, the question about the respective proportion of solutions remains open. 
\end{remark}

\begin{remark}
With \verb|Walnut|, we can also build a DFAO
computing the minimum (resp., maximum) possible sum of a length-$n$ block occurring in $\infw{tr}$.
Furthermore, for each $n$,
every possible sum between these two
extremes actually occurs for some
length-$n$ factor in $\infw{tr}$.
\end{remark} 

\subsubsection{The generalized Thue--Morse word on three letters}

We introduce a family of words over three letters that are closed to a generalization of the Thue--Morse word.

\begin{definition}
\label{def:l-m-TM}
    Let $\ell,m$ be integers such that $1\leq \ell < m$.
    The \emph{$(\ell,m)$-Thue--Morse word} $\infw{t}_{\ell,m}$ is the fixed point of the morphism $0\mapsto 0 \ell m$, $\ell \mapsto \ell m 0$, $m \mapsto m 0 \ell$. 
\end{definition}

In the case where $\ell=1$ and $m=2$, we find the so-called \emph{ternary Thue--Morse word} $\infw{t}_3$, which is the fixed point of the morphism $0\mapsto 012$, $1\mapsto 120$, $2\mapsto 201$. This word is a natural generalization of the ubiquitous Thue--Morse sequence, since it corresponds to the sum-of-digit function in base $3$, taken mod $3$.

\begin{theorem}[{\cite[Thm.~4.1]{KK-2017}}]
\label{thm:ab compl ternary TM}
Consider the ternary Thue--Morse word $\infw{t}_3$, i.e., the fixed point of the morphism $0\mapsto 012$, $1\mapsto 120$, $2\mapsto 201$.
The abelian complexity function $\rho^{\ab}_{\infw{t}_3}$ is the periodic infinite word $13(676)^{\omega}$. 
\end{theorem}

\begin{theorem}
\label{thm:ad compl ternary TM}
Consider the ternary Thue--Morse word $\infw{t}_3$, i.e., the fixed point of the morphism $0\mapsto 012$, $1\mapsto 120$, $2\mapsto 201$.
    The additive complexity function $\rho^{\add}_{\infw{t}_3}$ is the periodic infinite word $135^\omega$.   
\end{theorem}

\begin{proof}
The following \verb|Walnut| provides a linear representation of size $138$ for the additive complexity of $\infw{t}_3$: 

\begin{verbatim}
morphism h "0->012 1->120 2->201": 
promote TMG h:

def tmgPref0 "?msd_3 Er,t n=3*t+r & r<3 & (r=0 => s=t) 
    & ((r=1 & TMG[n-1]=@0) => s=t+1) 
    & ((r=1 & (TMG[n-1]=@1 | TMG[n-1]=@2)) => s=t) 
    & ((r=2 & (TMG[n-1]=@0 | TMG[n-1]=@1)) => s=t+1) 
    & ((r=2 & TMG[n-1]=@2) => s=t)":
def tmgPref1 "?msd_3 Er,t n=3*t+r & r<3 & (r=0 => s=t) 
    & ((r=1 & TMG[n-1]=@1) => s=t+1) 
    & ((r=1 & (TMG[n-1]=@0 | TMG[n-1]=@2)) => s=t) 
    & ((r=2 & (TMG[n-1]=@1 | TMG[n-1]=@2)) => s=t+1) 
    & ((r=2 & TMG[n-1]=@0) => s=t)":
def tmgPref2 "?msd_3 Eq,r $tmgPref0(n,q) & $tmgPref1(n,r) & q+r+s=n":

def tmgFac0 "?msd_3 Et,u $tmgPref0(i+n,t) & $tmgPref0(i,u) & s+u=t":
def tmgFac1 "?msd_3 Et,u $tmgPref1(i+n,t) & $tmgPref1(i,u) & s+u=t":
def tmgFac2 "?msd_3 Et,u $tmgPref2(i+n,t) & $tmgPref2(i,u) & s+u=t":

def tmgAddFacEq "?msd_3 Ep,q,r,s $tmgFac1(i,n,p) & $tmgFac2(i,n,q) 
    & $tmgFac1(j,n,r) & $tmgFac2(j,n,s) & p+2*q=r+2*s":

eval tmgAddCompRepLin n "?msd_3 Aj j<i => ~$tmgAddFacEq(i,j,n)":
\end{verbatim} 

The end of the proof is the same as for~\cref{thm:tribo}. The size of the minimal linear representation is $13$ and the semigroup trick algorithm terminates and produces the $3$-state DFAO of~\Cref{fig:automaton-AddTM}. The result follows immediately. 
\begin{figure}[h] 
\begin{center}
\begin{tikzpicture}
\tikzstyle{every node}=[shape=circle,fill=none,draw=black,minimum size=30pt,inner sep=2pt]
\node(a) at (0,0) {$q_0/1$};
\node(b) at (3,0) {$q_1/3$};
\node(c) at (6,0) {$q_2/5$};

\tikzstyle{every node}=[shape=circle,fill=none,draw=none,minimum size=10pt,inner sep=2pt]
\node(a0) at (-1,0) {};

\tikzstyle{every path}=[color =black, line width = 0.5 pt]
\tikzstyle{every node}=[shape=rectangle,minimum size=5pt,inner sep=2pt]
\draw [->] (a0) to [] node [] {}  (a);

\draw [->] (a) to [loop above] node [] {$0$}  (a);
\draw [->] (a) to [] node [above] {$1$}  (b);
\draw [->] (a) to [bend right] node [above] {$2$}  (c);

\draw [->] (b) to [] node [above] {$0,1,2$}  (c);

\draw [->] (c) to [loop above] node [] {$0,1,2$}  (c);
\end{tikzpicture}
\caption{A DFAO computing the additive complexity of the $(1,2)$-Thue--Morse word.}
\label{fig:automaton-AddTM}
\end{center}
\end{figure}
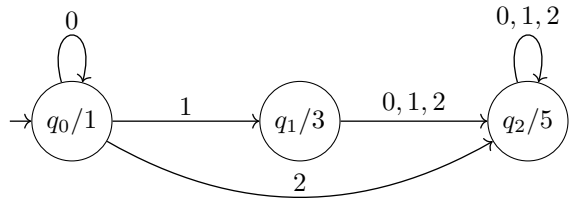
\end{proof} 

Changing the letters $\ell$ and $m$ does not modify the abelian complexity, so for all $1\leq\ell<m$, we have $\rho^{\ab}_{\infw{t}_{\ell,m}}=\rho^{\ab}_{\infw{t}_3}$. However, additive complexity might change over a different alphabet. In the particular case where $\ell=1$ and $m=2$, the following gives an alternative proof of~\cref{thm:ad compl ternary TM} with only combinatorial tools.
Note that the statement on $\rho^{\ab}_{\infw{t}_{\ell,m}}$ was also proven in \cite{KK-2017}, but we provide here a simpler and more concise proof. 

\begin{theorem} \label{Thue--Morse}
Let $\ell,m$ be integers such that $1\leq \ell < m$.
Consider the $(\ell,m)$-Thue--Morse word, i.e., the fixed point of the morphism $0\mapsto 0 \ell m$, $\ell \mapsto \ell m 0$, $m \mapsto m 0 \ell$. 
Then its abelian complexity satisfies $\rho^{\ab}_{\infw{t}_{\ell,m}}=136(766)^{\omega}$ and its additive complexity satisfies $\rho^{\add}_{\infw{t}_{\ell,m}} = \rho^{\ab}_{\infw{t}_{\ell,m}}$ if $m\neq 2\ell $, and $\rho^{\add}_{\infw{t}_{\ell,m}}=135^{\omega}$ if $m=2\ell$.
\end{theorem}

\begin{proof}
We clearly have $\rho^{\add}_{\infw{t}_{\ell,m}}(0)=1$, $\rho^{\add}_{\infw{t}_{\ell,m}}(1)=3$,  and $\rho^{\add}_{\infw{t}_{\ell,m}}(2)$ is equal to $5$ or $6$ depending on whether $m=2\ell$ or not.
We examine length-$n$ factors of $\infw{t}_{\ell,m}$ for $n\ge 2$.
Each such factor can be written as $y=pf(x)s$ where $f$ is the morphism $0\mapsto 0 \ell m$, $\ell \mapsto \ell m 0$, $m \mapsto m 0 \ell$ of~\cref{def:l-m-TM} and $p$ (resp., $s$) is a proper suffix (resp., prefix) of an image $f(a)$  for $a\in\{0,\ell,m\}$.
In particular, note that $p,s\in\{\varepsilon,0,\ell,m,0\ell, \ell m, m0\}$.
In the following, we examine the \emph{weight} of $y$, which is the quantity $0\cdot |y|_0+\ell\cdot |y|_\ell+m\cdot |y|_m$.
More precisely, we count how many different weights $y$ can have, which in turn gives the number of different additive equivalence classes.

First assume that $|y|=3n$ for some $n\ge 1$.
Then we have two cases depending on whether $p,s$ are empty or not. If $p=s=\varepsilon$, then $|x|=n$ and this case corresponds to the first line of~\cref{tab:weights3n}. Otherwise, $|x|=n-1$ and $|ps|=3$. In that case, since the roles of $p$ and $s$ are symmetric when computing the weight of the factor, all the possible cases are depicted in~\cref{tab:weights3n}. 
\begin{table}[h]
    \centering
    \begin{tabular}{c|c|c|c|c|c}
     $p$ & $s$ & $|y|_0$ & $|y|_{\ell}$ & $|y|_{m}$ & $0\cdot |y|_0+\ell\cdot |y|_\ell+m\cdot |y|_m$ \\\hline \hline 
     $\varepsilon$ & $\varepsilon$ & $n$ & $n$ & $n$ & $\ell n + mn$ \\ \hline 
     $0$ & $0\ell$ & $n+1$ & $n$ & $n-1$ & $\ell n + m(n-1)$ \\ \hline 
     $0$ & $\ell m$ & $n$ & $n$ & $n$ & $\ell n + mn$ \\ \hline 
     $0$ & $m 0$ & $n+1$ & $n-1$ & $n$ & $\ell (n+1) + m(n-1)$ \\ \hline 
     $\ell$ & $0\ell $ & $n$ & $n+1$ & $n-1$ & $\ell n + m(n+1)$ \\ \hline 
     $\ell$ & $\ell m$ & $n-1$ & $n+1$ & $n$ & $\ell (n-1) + mn$ \\ \hline 
     $\ell$ & $m 0$ & $n$ & $n$ & $n$ & $\ell n + mn$ \\ \hline 
     $m$ & $0\ell$ & $n$ & $n$ & $n$ & $\ell n + mn$ \\ \hline 
     $m$ & $\ell m$ & $n-1$ & $n$ & $n+1$ & $\ell n + m(n+1)$ \\ \hline 
     $m$ & $m0$ & $n$ & $n-1$ & $n+1$ & $\ell (n-1) + m(n+1)$ \\ 
\end{tabular}
    \caption{The possible weights of factors of the $(\ell,m)$-Thue--Morse word $\infw{t}_{\ell,m}$ of the form $y=pf(x)s$ where $|y|=3n$ for some $n\ge 1$.}
    \label{tab:weights3n}
\end{table}
From the third, fourth and fifth columns of the table, we observe that there are seven different abelian classes (only the class where $|y|_0=|y|_{\ell}=|y|_m=n$ appears more than once) and this proves that $\rho^{\ab}_{\infw{t}_{\ell,m}}(3n)=7$. The corresponding weights of these seven abelian classes can be written as $(n-1)\cdot(\ell+m) + \delta$ with $\delta \in \{\ell, m, 2\ell, \ell+m, 2m, 2\ell+m, \ell+2m\}$. Since \begin{align*}
    \ell < \min\{m,2\ell\} \le &\max\{m,2\ell\} < \ell +m \\ &<\min\{2m,2\ell+m\} \le \max\{2m,2\ell+m\} < \ell+2m, 
    \end{align*} this now proves that $\rho^{\add}_{\infw{t}_{\ell,m}}(3n)$ is equal to $7$ if $m\neq 2\ell $, and to $5$ otherwise.


Assume that $|y|=3n+1$ for some $n\ge 1$. With similar reasoning and up to the symmetry between $p$ and $s$, we list all cases in~\cref{tab:weights3n+1}.
\begin{table}[h]
    \centering
    \begin{tabular}{c|c|c|c|c|c}
     $p$ & $s$ & $|y|_0$ & $|y|_{\ell}$ & $|y|_{m}$ & $0\cdot |y|_0+\ell\cdot |y|_\ell+m\cdot |y|_m$ \\\hline \hline 
     $0$ & $\varepsilon$ & $n+1$ & $n$ & $n$ & $\ell n + mn$ \\ \hline 
     $\ell$ & $\varepsilon$ & $n$ & $n+1$ & $n$ & $\ell (n+1) + mn$ \\ \hline 
     $m$ & $\varepsilon$ & $n$ & $n$ & $n+1$ & $\ell n + m(n+1)$ \\ \hline 
     $0\ell$ & $0\ell$ & $n+1$ & $n+1$ & $n-1$ & $\ell (n+1) + m(n-1)$ \\ \hline 
     $0\ell$ & $\ell m$ & $n$ & $n+1$ & $n$ & $\ell (n+1) + mn$ \\ \hline 
     $0\ell$ & $m0$ & $n+1$ & $n$ & $n$ & $\ell n + mn$ \\ \hline 
     $\ell m$ & $\ell m$ & $n-1$ & $n+1$ & $n+1$ & $\ell (n-1) + m(n+1)$ \\ \hline 
     $\ell m$ & $m0$ & $n$ & $n$ & $n+1$ & $\ell n + m(n+1)$ \\ \hline 
     $m0$ & $m0$ & $n+1$ & $n-1$ & $n+1$ & $\ell (n-1) + m(n+1)$ \\ 
\end{tabular}
    \caption{The possible weights of factors of the $(\ell,m)$-Thue--Morse word $\infw{t}_{\ell,m}$ of the form $y=pf(x)s$ where $|y|=3n+1$ for some $n\ge 1$.}
    \label{tab:weights3n+1}
\end{table}
Therefore, we see that there are six different abelian classes, which proves that $\rho^{\ab}_{\infw{t}_{\ell,m}}(3n+1)=6$. The corresponding weights of these six abelian classes can be written as $(n-1)\cdot(\ell+m) + \delta$ with $\delta \in \{2\ell, \ell+m, 2m, 2\ell+m, \ell+2m,2\ell+2m\}$. Since
\[
2\ell < \ell+m < \min\{2\ell+m,2m\} \le \max\{2\ell+m,2m\} < \ell+2m <2\ell+2m,
\]
we have that $\rho^{\add}_{\infw{t}_{\ell,m}}(3n+1)$ is equal to $6$ if $m\neq 2\ell $, and to $5$ otherwise.

Finally, assume that $|y|=3n+2$ for some $n\ge 1$. Up to the symmetry between $p$ and $s$, we list all cases in~\cref{tab:weights3n+2}.
\begin{table}[h]
    \centering
    \begin{tabular}{c|c|c|c|c|c}
     $p$ & $s$ & $|y|_0$ & $|y|_{\ell}$ & $|y|_{m}$ & $0\cdot |y|_0+\ell\cdot |y|_\ell+m\cdot |y|_m$ \\\hline \hline 
     $0$ & $0$ & $n+2$ & $n$ & $n$ & $\ell n + mn$ \\ \hline 
     $0$ & $\ell$ & $n+1$ & $n+1$ & $n$ & $\ell (n+1) + mn$ \\ \hline 
     $0$ & $m$ & $n+1$ & $n$ & $n+1$ & $\ell n + m(n+1)$ \\ \hline 
     $\ell$ & $\ell$ & $n$ & $n+2$ & $n$ & $\ell (n+2) + mn$ \\ \hline 
     $\ell$ & $m$ & $n$ & $n+1$ & $n+1$ & $\ell (n+1) + m(n+1)$ \\ \hline 
     $m$ & $m$ & $n$ & $n$ & $n+2$ & $\ell n + m(n+2)$ \\ \hline 
     $0\ell $ & $\varepsilon$ & $n+1$ & $n+1$ & $n$ & $\ell (n+1) + mn$ \\ \hline 
     $\ell m$ & $\varepsilon$ & $n$ & $n+1$ & $n+1$ & $\ell (n+1) + m(n+1)$ \\ \hline 
     $m0$ & $\varepsilon$ & $n+1$ & $n+1$ & $n$ & $\ell (n+1) + mn$ \\ 
\end{tabular}
    \caption{The possible weights of factors of the $(\ell,m)$-Thue--Morse word $\infw{t}_{\ell,m}$ of the form $y=pf(x)s$ where $|y|=3n+2$ for some $n\ge 0$.}
    \label{tab:weights3n+2}
\end{table}
Once again, we see that there are six different abelian classes, so $\rho^{\ab}_{\infw{t}_{\ell,m}}(3n+2)=6$. The corresponding weights of these six abelian classes can be written as $n\cdot(\ell+m) + \delta$ with $\delta \in \{0,\ell,m,2\ell,\ell+m, 2m\}$. Since 
\[ 0<\ell < \min\{2\ell,m\} \le \max\{2\ell,m\} < \ell +m <2m,\] 
we have that $\rho^{\add}_{\infw{t}_{\ell,m}}(3n+2)$ is equal to $6$ if $m\neq 2\ell $, and to $5$ otherwise.
\end{proof}

\subsection{Bounded additive and unbounded abelian complexities: a variant of the Thue--Morse word}
\label{sec:bounded add and unbounded ab}

Thue introduced a variation of his sequence that is sometimes called the \emph{ternary squarefree Thue--Morse word}, and abbreviated as $\infw{vtm}$ (the letter ``v'' stands for ``variant''). It is the sequence~\cite[A036577] {Sloane} in the OEIS;  for more on the word $\infw{vtm}$, see~\cite{Berstel1979}.

\begin{definition}[Variant of Thue--Morse]
    We let $\infw{vtm}$ be the fixed point of $f:0\mapsto 012, 1\mapsto 02, 2\mapsto 1$, starting with $0$.
\end{definition}

The abelian complexity of the variant of the Thue--Morse word is unbounded. 

\begin{theorem}[{\cite[Cor.~1]{BSCRF-2014}}]
    Let $\infw{vtm}$ be the fixed point of $f:0\mapsto 012, 1\mapsto 02, 2\mapsto 1$, starting with $0$.
    Its abelian complexity is $O(\log n)$ with constant approaching $3/4$ (assuming base-$2$ logarithm), and it is $\Omega (1)$ with constant $3$.
\end{theorem}

However, we prove that the additive complexity of the word $\infw{vtm}$ is bounded.

\begin{theorem} \label{thm:vtm}
 Let $\infw{vtm}$ be the fixed point of $f:0\mapsto 012, 1\mapsto 02, 2\mapsto 1$, starting with $0$.
 Its additive complexity is the periodic infinite word $13^\omega$.
\end{theorem}

\begin{proof}
    Let $n\ge 1$ and $x\in \mathcal{L}_n(\infw{vtm})$. Let us prove that $\sum_{a=0}^2 a\cdot |x|_a  \in\{n-1,n,n+1\}$.
    Write $x=pf(y)s$ where $p$ (resp., $s$) is a proper suffix (resp., prefix) of an image $f(a)$, $a\in\{0,1,2\}$. Then we have $p\in \{\eps,12,2\}$ and $s\in \{\eps,0,01\}$. 
    By definition of the morphism $f$, observe that $|f(y)|_2=|f(y)|_0$.
    Therefore, depending on the words $p$ and $s$, $|x|_2=|x|_0+c$ with $c\in\{-1,0,1\}$, which suffices since $|x|_0+|x|_1+|x|_2=n$.   
\end{proof}

Therefore, the word $\infw{vtm}$ has unbounded abelian complexity and  bounded additive complexity; also see~\cref{fig:add-complexity-VTM}.
In particular,~\cite[Lemma~3]{RSZ-2011} implies that $\infw{vtm}$ cannot be balanced, so there exist non-balanced infinite words with bounded additive complexity.
Another example exhibiting the same behavior for its abelian and additive complexity is given in~\cite{ABJS-2012}.

\begin{figure}[h]
    \centering
    \includegraphics[scale=0.45]{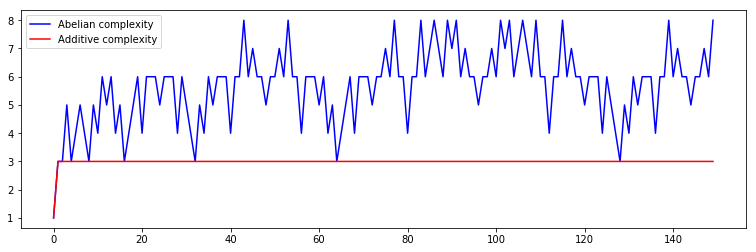}
    \caption{The first few values of the abelian and additive complexities for the variant of the Thue--Morse word.}
    \label{fig:add-complexity-VTM}
\end{figure}

\subsection{Unbounded additive and abelian complexities}

In this short section, we exhibit a word such that both its additive and abelian complexities are both unbounded. 

\begin{theorem}[{\cite[Thm.~1 and Cor.~1]{CWW-2019}}]
    Let $\infw{x}$ be the fixed point of the Thue--Morse-like morphism $0 \mapsto 01$, $1 \mapsto 12$, $2 \mapsto 20$.
    Then $\rho^{\add}_{\infw{x}}(n) = 2\lfloor \log_2 n \rfloor + 3$ for all $n\ge 1$.
    In particular, the sequence $(\rho^{\add}_{\infw{x}}(n))_{n\ge 0}$ is $2$-regular. 
\end{theorem}


Recall that, for an integer $k\ge 1$, a word $w$ is an \emph{abelian $k$-power} if we can write $w=x_1 x_2 \cdots x_k$ where each $x_i$, $i\in\{1,\ldots,k\}$, is a permutation of $x_1$.
For instance, $\texttt{reap}\cdot \texttt{pear}$ and $\texttt{de}\cdot\texttt{ed}\cdot\texttt{ed}$ are respectively an abelian square and cube in English.
Similarly, $w$ is an \emph{additive $k$-power} if we can write $w=x_1 x_2 \cdots x_k$ with $|x_i|=|x_1|$ for all $i\in\{1,\ldots,k\}$ and $x_1 \sim_{\add} x_2 \sim_{\add} \cdots \sim_{\add} x_k$.
The length of each $x_i$, $i\in\{1,\ldots,k\}$, is called the \emph{order} of $w$. As mentioned in the introduction, the following result is one of the main results known on additive complexity. 

\begin{theorem}[{\cite[Thm.~2.2]{ABJS-2012}}]
\label{thm:bounded-additive-complexity-kpower}
    Let $\infw{x}$ be an infinite word over a finite subset of $\Z$.
    If $\rho^{\add}_{\infw{x}}$ is bounded, then $\infw{x}$ contains an additive $k$-power for every positive integer $k$.
\end{theorem}

\begin{proposition}
    Let $\infw{w}$ be the fixed point of the morphism $0 \mapsto 03$, $1 \mapsto 43$, $3 \mapsto 1$, $4 \mapsto 01$.
    Then $\rho^{\add}_{\infw{w}}$ is unbounded.
\end{proposition}

\begin{proof}
    In~\cite{CCSS2014} it is shown that $\infw{w}$ is additive-cube-free.
    The result then follows from~\cref{thm:bounded-additive-complexity-kpower}.    
\end{proof}

However, for the latter word $\infw{w}$, it seems interesting to study $\rho^{\ab}_{\infw{w}}-\rho^{\add}_{\infw{w}}$, since these two complexities are very close. Indeed, surprisingly the first time these two complexities are different appears at $n=23$, as the two factors $11011031430110343430314$ and $30310110110314303434303$ are additively but not abelian equivalent. Also, notice that every additive square of the word $\infw{w}$ is an abelian square~\cite[Theorem 5.1]{CCSS2014}. Together with the fact that this word is additive-cube-free, it shows that abelian and additive properties of this word are relatively close. Indeed, \cref{fig:abad-comp-CCSS} illustrates that the values of the difference between the additive and abelian complexity functions is close to $0$. This motivates the study of the next section. 
    
\begin{figure}[h]
    \centering
    \includegraphics[scale=0.45]{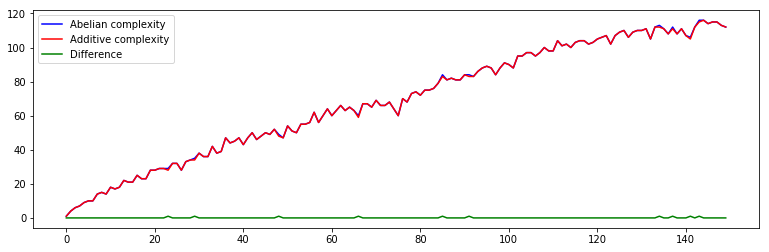}
    \caption{The first few values of the abelian and additive complexities as well as their difference for the fixed point of the morphism $0 \mapsto 03$, $1 \mapsto 43$, $3 \mapsto 1$, $4 \mapsto 01$.}
    \label{fig:abad-comp-CCSS}
\end{figure}

\section{Equality between abelian and additive complexities}
\label{sec:equality between add and ab}

It is clear that abelian complexity does not depend on the values of the alphabet, in contrast with additive complexity. A map $v:A\rightarrow \mathbb{N}$ is called a \emph{valuation} over an alphabet $A$. One might consider the following question.

\begin{question}
Given an alphabet $A$, is there a valuation such that the additive complexity of a given sequence is equal to its abelian complexity?
\end{question}

For instance for the word $\infw{vtm}$, defined originally over the alphabet $\{0,1,2\}$, we have already proved in~\Cref{thm:vtm} that $\rho^{\add}_{\infw{vtm}}(n)=3$ for all $n \geq 1$. However, over the alphabet $\{0,1,3\}$, i.e.,  changing $2$ into $3$, (resp., $\{0,1,4\}$), one can easily check that the first time that the additive and abelian complexities are not equal is for $n=11$ (resp., $n=43$). But, over the alphabet $\{0,1,5\}$, we have observed that both complexities are equal up to $n=50000$. The main idea is that if the value of a letter is sufficiently large compared to the other values, then two additively equivalent factors are also abelian equivalent. Using this idea, we prove the following theorem. 

\begin{theorem}
\label{thm:vtm-egality-add-ab}
Consider the fixed point $\infw{vtm}_{\lambda}$ of the morphism $f_{\lambda} :0\mapsto 01\lambda, 1\mapsto 0\lambda, \lambda\mapsto 1$, where $\lambda$ is a non-negative integer and $\lambda\geq 2$. For all $\lambda \geq 5$,  we have $\rho^{\add}_{\infw{vtm}_\lambda}(n)=\rho^{\ab}_{\infw{vtm}_\lambda}(n)$. 
\end{theorem}

\begin{proof}
    By~\cref{lem:comparison-abb-ad}, it is sufficient to prove that two factors are additively equivalent if and only if they are abelian equivalent. 

    Let $x,y \in \mathcal{L}_n(\infw{vtm}_{\lambda})$ such that $x\sim_{\add}y$. Write $x=pf_\lambda(x')s$ where $p$ (resp., $s$) is a proper suffix (resp., prefix) of an image $f_\lambda(a)$ with $a\in\{0,1,\lambda\}$. Observe that $p\in \{\eps,\lambda,1\lambda\}$ and $s\in \{\eps,0,01\}$. 
    Also, by definition of the morphism $f_\lambda$, we have $|f_\lambda(x')|_\lambda=|f_\lambda(x')|_0$.
    Therefore, depending on the words $p$ and $s$, we have $|x|_\lambda=|x|_0+c_x$ for some $c_x\in\{-1,0,1\}$. In a similar way, we have $|y|_\lambda=|y|_0+c_y$ for some $c_y\in\{-1,0,1\}$. 
    By the assumption that $x\sim_{\add}y$, we have $0|x|_0+1|x|_1+\lambda|x|_\lambda=0|y|_0+1|y|_1+\lambda|y|_\lambda$.
    From the previous observations, we may write this equality as \[|x|_0+|x|_1+|x|_\lambda+(\lambda-2)|x|_\lambda+c_x=|y|_0+|y|_1+|y|_\lambda+(\lambda-2)|y|_\lambda+c_y.\] Since $|x|_0+|x|_1+|x|_\lambda=n=|y|_0+|y|_1+|y|_\lambda$, we have $(\lambda-2)(|x|_\lambda-|y|_\lambda)=c_y-c_x$.
    However, $c_y-c_x \in \{-2,\ldots,2\}$ implies that $|x|_\lambda=|y|_\lambda$ and $c_y=c_x$,  since $\lambda-2\geq 3$. Thus, $|x|_0=|y|_0$.
    Since $|x|=n=|y|$, we also have $|x|_1=|y|_1$, and then that $x$ and $y$ are abelian equivalent. 
    This ends the proof.
\end{proof}

For $C$-balanced words over an alphabet of fixed size $k$, we prove that it is always possible to find a valuation for the alphabet such that the additive complexity is the same as the abelian complexity. 

\begin{theorem}\label{thm:balanced_AB=AD}
Let $k,C\ge 1$ be two integers.
There exists an alphabet $\Sigma \subset \mathbb{N}$ of size $k$ such that, for each $C$-balanced word $\infw{w}$ over $\Sigma$, we have $\rho^{\add}_{\infw{w}}=\rho^{\ab}_{\infw{w}}$. 
\end{theorem}

\begin{proof}
    Such as in the proof of~\cref{thm:vtm-egality-add-ab}, we prove that over the alphabet $\Sigma$, two additively equivalent same-length factors of $\infw{w}$ are also abelian equivalent. Let $\Sigma=\{a_1,\ldots,a_k\}$ be a subset of $\N$ such that
\begin{align}\label{eq:cond on the letters}
    a_1=0, a_2=1
\quad \text{and} \quad
(a_1+\cdots+a_{j-1})C<a_j \quad \text{for all } 2\leq j \leq k.
\end{align}
    Now take $x,y \in \mathcal{L}_n(\infw{w})$ with $x\sim_{\add}y$. This condition can be rewritten as
    \begin{align} \label{eq:firsteq}
        a_1(|x|_{a_1}-|y|_{a_1})+\cdots+a_{k-1}(|x|_{a_{k-1}}-|y|_{a_{k-1}})=a_k(|y|_{a_k}-|x|_{a_k}).
    \end{align}
    Observe that the balancedness of $\infw{w}$ together with Inequalities~\eqref{eq:cond on the letters} imply that the left-hand side of Equality~\eqref{eq:firsteq} belongs to the set $\{-a_k +1, \ldots, a_k-1\}$.
    Since the right-hand side of Equality~\eqref{eq:firsteq} is a multiple of $a_k$, we must have
    \begin{equation}\label{eq:seceq}
    \left\{
    \begin{array}{l}
    |x|_{a_k}=|y|_{a_k}, \\
    a_1(|x|_{a_1}-|y|_{a_1})+\cdots+a_{k-1}(|x|_{a_{k-1}}-|y|_{a_{k-1}})=0.
    \end{array}
    \right.
    \end{equation}
    Using similar reasoning, replacing Equality~\eqref{eq:firsteq} with Equalities~\eqref{eq:seceq}, we deduce that $|x|_{a_{k-1}}=|y|_{a_{k-1}}$.
    Continuing in this fashion, we prove that $|x|_a=|y|_a$ for every $a\in \Sigma$, which is enough.   
\end{proof}

\begin{remark} Since the Tribonacci word $\infw{tr}$ is $2$-balanced, \Cref{thm:balanced_AB=AD} implies that over the alphabet $\{0,1,3\}$, its additive complexity is equal to its abelian complexity. 
\end{remark}





\section{Abelian and additive powers}

In~\cite{Fici-etal-2016a,Fici-etal-2016b}, abelian powers of Sturmian words were examined.
In particular, the following result was obtained for the Fibonacci word $\infw{f}=010010100100101001010 \cdots$, which is the fixed point of the morphism $0\mapsto 01, 1\mapsto 0$.
Also, see the sequence~\cite[A336487]{Sloane} in the OEIS.

\begin{proposition}[{\cite{Fici-etal-2016a,Fici-etal-2016b}}]
\label{pro:Fici's criterion}
Let $k\ge 1$ be an integer and consider the Fibonacci word $\infw{f}$, i.e., the fixed point of the morphism $0\mapsto 01, 1\mapsto 0$. Then $\infw{f}$ has an abelian $k$-power of order $n$ if and only if $\lfloor k \varphi n\rfloor \equiv \modd{0,-1} {k}$, where $\varphi = \frac{1+\sqrt{5}}{2}$ is the golden ratio.
\end{proposition}



For instance, when $k=3$, we can compute an 11-state DFA accepting, in Fibonacci representations, exactly those $n$ for which there is an abelian cube of order $n$ in $\infw{f}$.

As Arnoux-Rauzy and episturmian sequences generalize Sturmian sequences, it is quite natural to try to understand the orders of abelian and additive powers in these sequences.
An archetypical example is the Tribonacci word $\infw{tr}$ (recall~\cref{sec: Tribonacci}). 
We obtain the following results on squares and cubes using \texttt{Walnut} and the fact that the frequency of each letter $0,1,2$ in $\infw{tr}$ is Tribonacci-synchronized (see~\cite[\S~10.12]{Walnut2} and/or~\cref{sec: Tribonacci}). 

\begin{theorem}[{\cite[Thm.~10.13.5]{Walnut2}}]    
\label{thm:abelian square in Tribonacci}
Let $\infw{tr}$ be the Tribonacci word, i.e., the fixed point of the morphism $0\mapsto 01$, $1\mapsto 02$.
    There are abelian squares of all orders in $\infw{tr}$.
Furthermore, if we consider two abelian squares $xx'$ and $yy'$ to be equivalent if $x \sim_{\ab} y$, then every order has either one or two abelian squares.  Both possibilities occur infinitely often.
    \end{theorem}

\begin{theorem}   
\label{thm: additive cubes in Tribonacci}
Let $\infw{tr}$ be the Tribonacci word, i.e., the fixed point of the morphism $0\mapsto 01$, $1\mapsto 02$.
    There is a (minimal) Tribonacci automaton of $1169$ (resp., $4927$) states recognizing the Tribonacci representation of those $n$ for which there is an abelian (resp., additive) cube of order $n$ in $\infw{tr}$.
\end{theorem}

\begin{proof}
For the part about abelian cubes, see~\cite[p.~295]{Walnut2}. See also the respective sequences~\cite[A345717,A347752]{Sloane} in the OEIS. For the additive cubes, we can determine the orders of additive cubes in $\infw{tr}$ with the following function:
\begin{verbatim} def tribAddCube "?msd_trib Ei $tribAddFacEq(i,i+n,n)
    & $tribAddFacEq(i,i+2*n,n)":
\end{verbatim}
where \texttt{tribAddFacEq} is the function defined in the proof of \Cref{thm:tribo}. This leads to a Tribonacci automaton of 4927 states. 
\end{proof}

We also note that~\cref{thm:tribo,thm:bounded-additive-complexity-kpower} imply the existence of additive $k$-powers in $\infw{tr}$ for all $k\ge 1$.
When $k=2$, additive squares exist for all orders by~\cref{thm:abelian square in Tribonacci}.
For $k=3$, orders of additive cubes are given in~\cref{thm: additive cubes in Tribonacci} by a large automaton, and no simple description seems to be possible. When $k=4$, the same procedure on \verb|Walnut| requires a much larger memory, and it appears that a simple desk computer cannot achieve it.
We naturally wonder about larger powers and leave the following as a relatively difficult open question.

\begin{problem}
 Characterize the orders of additive $k$-powers in the Tribonacci word $\infw{tr}$.
\end{problem}

It is shown in~\cite{CFZ-2000} that the behavior of the abelian complexity of Arnoux-Rauzy words might be erratic.
In particular, there exist such words with unbounded abelian complexity.
We leave open the research direction of studying the additive complexity of such words and episturmian sequences.
For instance, is there a result similar to~\cref{pro:Fici's criterion} in the framework of additive powers?


\appendix 

\section{Semigroup trick} \label{Appendix:semigroup trick}

In this section, we give more details about the semigroup trick algorithm discussed in \cref{rk: constructive proof vs semi-group trick halts}, as well as a simple example. For more details, we refer to~\cite[\S~4.11]{Walnut2}. 

Suppose we are given a linear representation of a regular sequence $\infw{x}$ in a positional numeration $U$, i.e., there exist a column vector $\lambda$, a row vector $\gamma$ and a matrix-valued morphism $\mu$ such that $\infw{x}(n)=\lambda \mu(\rep_U(n)) \gamma$. If we suspect that $\infw{x}$ takes on only finitely many values, i.e., $\infw{x}$ is automatic, one can apply the \texttt{Semigroup Trick} algorithm presented in~\cite[\S~4.11]{Walnut2}. This algorithm explores the tree of possibilities for the vectors $\lambda \mu(x)$ for $x\in \Sigma^{*}$ using breadth-first search until no new vector is generated. If the search halts, then the semigroup $\{\lambda \mu(x): x\in \Sigma^*\}$ is finite. Furthermore, the algorithm constructs a DFAO computing $\infw{x}$ by letting the states be the set of distinct vectors that are reachable, the initial state be $\lambda$, and the output function associated with each vector $w$ be $w\gamma$.

\begin{example}
Consider the following linear representation \begin{align*}
\lambda=\begin{pmatrix}
1 & 0 & 0 & 0\\
\end{pmatrix}, \; \mu(0)=\begin{pmatrix}
1 & 0 & 0 & 0\\
1 & 0 & 0 & 0\\
1 & 0 & 0 & 0\\
0 & 0 & 0 & 1\\
\end{pmatrix}, \; \mu(1)=\begin{pmatrix}
0 & 0 & 1 & 0\\
0 & 0 & 1 & 0\\
0 & 0 & -1 & 1\\
0 & 0 & 0 & 1\\
\end{pmatrix}, \; \gamma=\begin{pmatrix}
0 \\ 0 \\ 0 \\ 1
\end{pmatrix}. 
\end{align*} The first values of the sequence are $0,0,0,1,0,0,1,0,0,0,0,1,1,1,0,1,\ldots$. The steps of the semigroup trick algorithm are the following: \begin{enumerate}
\item We start with the vector $\lambda$ and we compute $\lambda \mu(0)$ and $\lambda \mu(1)$. We have $\lambda \mu(0)=\lambda$ and $\lambda \mu(1)=\begin{pmatrix}
0 \; 0 \; 1 \; 0\\
\end{pmatrix}=w_1$, which is a new state. Therefore we add the transitions $\lambda \xrightarrow{0}\lambda$ and $\lambda \xrightarrow{1} w_1$ in the automaton. We add $w_1$ to the queue.

\item We compute $w_1 \mu(0)$ and $w_1 \mu(1)$. We have $w_1 \mu(0)=\lambda$ and $w_1 \mu(1)=\begin{pmatrix}
0 \; 0 \; -1 \; 0\\
\end{pmatrix}=w_2$. Therefore we add the transitions $w_1 \xrightarrow{0}\lambda$ and $w_1 \xrightarrow{1} w_2$ in the automaton. We add $w_2$ to the queue. 

\item We compute $w_2 \mu(0)$ and $w_2 \mu(1)$. We have $w_2 \mu(0)=\begin{pmatrix}
-1 \; 0 \; 0 \; 0\\
\end{pmatrix}=w_3$ and $w_2 \mu(1)=w_1$. Therefore we add the transitions $w_2 \xrightarrow{0}w_3$ and $w_2 \xrightarrow{1} w_1$ in the automaton. We add $w_3$ to the queue. 

\item We compute $w_3 \mu(0)$ and $w_3 \mu(1)$. We have $w_3 \mu(0)=w_3$ and $w_3 \mu(1)=w_2$. Therefore we add the transitions $w_3 \xrightarrow{0}w_3$ and $w_3 \xrightarrow{1} w_2$ in the automaton. Since there is no new state, the algorithm halts. 

\item For each state $w$, we compute the value $w \gamma$. We have $\lambda \gamma=0$, $w_1 \gamma=0$, $w_2 \gamma=1$ and $w_3 \gamma=1$, which are the outputs in the DFAO.
\end{enumerate}
Finally, the obtained DFAO for the sequence $\infw{x}$ is given in~\cref{fig:Fibonacci-DFAO}.

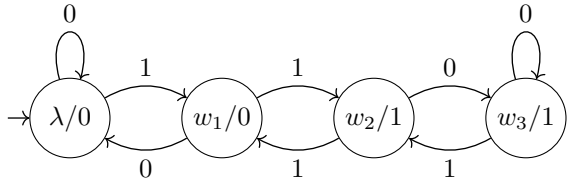
\begin{figure}[ht]
\begin{center}
\begin{tikzpicture}
\tikzstyle{every node}=[shape=circle,fill=none,draw=black,minimum size=30pt,inner sep=2pt]
\node(a) at (0,0) {$\lambda/0$};
\node(b) at (2,0) {$w_1/0$};
\node(c) at (4,0) {$w_2/1$};
\node(d) at (6,0) {$w_3/1$};

\tikzstyle{every node}=[shape=circle,fill=none,draw=none,minimum size=10pt,inner sep=2pt]
\node(a0) at (-1,0) {};

\tikzstyle{every path}=[color =black, line width = 0.5 pt]
\tikzstyle{every node}=[shape=circle,minimum size=5pt,inner sep=2pt]
\draw [->] (a0) to [] node [] {}  (a);

\draw [->] (a) to [loop above] node [] {$0$}  (a);
\draw [->] (a) to [bend left] node [above] {$1$}  (b);

\draw [->] (b) to [bend left] node [below] {$0$}  (a);
\draw [->] (b) to [bend left] node [above] {$1$}  (c);

\draw [->] (c) to [bend left] node [below] {$1$}  (b);
\draw [->] (c) to [bend left] node [above] {$0$}  (d);

\draw [->] (d) to [loop above] node [] {$0$}  (d);
\draw [->] (d) to [bend left] node [below] {$1$}  (c);

;
\end{tikzpicture}
\caption{An example of DFAO obtained with the semigroup trick algorithm.}
\label{fig:Fibonacci-DFAO}
\end{center}
\end{figure}

Now we recognize the automaton of the famous Rudin--Shapiro sequence. Notice that from a minimal linear representation, the semigroup trick algorithm halts if and only if the sequence is bounded. Since \texttt{Walnut} does not provide a minimal linear representation, it may be not sufficient to use the semigroup trick algorithm without minimization. However, for the examples given in this paper, the semigroup trick algorithm halts even before minimization but the obtained automata are not minimal. 
\end{example}


\section*{Acknolwedgments}
Pierre Popoli is supported by ULiège's Special Funds for Research, IPD-STEMA Program. Jeffrey Shallit is supported by NSERC grants 2018-04118 and 2024-03725. Manon Stipulanti is an FNRS Research Associate supported by the Research grant 1.C.104.24F

We thank Matthieu Rosenfeld and Markus Whiteland for helpful discussions, and Eric Rowland for implementing useful Mathematica code.


\bibliographystyle{plainurl}
\bibliography{biblio.bib}

\begin{thebibliography}{10}

\bibitem{ACSS2024+}
Jean-Paul Allouche, John Campbell, Jeffrey Shallit, and Manon Stipulanti.
\newblock The reflection complexity of sequences over finite alphabets.
\newblock Preprint available at \url{https://arxiv.org/abs/2406.09302}.

\bibitem{Allouche-Dekking-Queffelec-2021}
Jean-Paul Allouche, Michel Dekking, and Martine Queff{\'e}lec.
\newblock Hidden automatic sequences.
\newblock {\em Comb. Theory}, 1:15, 2021.
\newblock Id/No 20.
\newblock \href {https://doi.org/10.5070/C61055386}
  {\path{doi:10.5070/C61055386}}.

\bibitem{AS03}
Jean-Paul Allouche and Jeffrey Shallit.
\newblock {\em Automatic sequences: Theory, applications, generalizations}.
\newblock Cambridge University Press, Cambridge, 2003.
\newblock \href {https://doi.org/10.1017/CBO9780511546563}
  {\path{doi:10.1017/CBO9780511546563}}.

\bibitem{AndradeMol-2024}
Jonathan Andrade and Lucas Mol.
\newblock Avoiding abelian and additive powers in rich words, 2024.
\newblock Preprint available at \url{https://www.arxiv.org/pdf/2408.15390}.

\bibitem{ABJS-2012}
Hayri Ardal, Tom Brown, Veselin Jungi\'{c}, and Julian Sahasrabudhe.
\newblock On abelian and additive complexity in infinite words.
\newblock {\em Integers}, 12(5):795--804, 2012.
\newblock \href {https://doi.org/10.1515/integers-2012-0005}
  {\path{doi:10.1515/integers-2012-0005}}.

\bibitem{Banero-2013}
Graham Banero.
\newblock On additive complexity of infinite words.
\newblock {\em J. Integer Seq.}, 16(1):Article 13.1.5, 20, 2013.

\bibitem{Berstel1979}
Jean Berstel.
\newblock Sur les mots sans carr\'{e} d\'{e}finis par un morphisme.
\newblock In {\em Automata, languages and programming ({S}ixth {C}olloq.,
  {G}raz, 1979)}, volume~71 of {\em Lecture Notes in Comput. Sci.}, pages
  16--25. Springer, Berlin-New York, 1979.

\bibitem{BerstelPerrin2007}
Jean Berstel and Dominique Perrin.
\newblock The origins of combinatorics on words.
\newblock {\em Eur. J. Comb.}, 28(3):996--1022, 2007.
\newblock \href {https://doi.org/10.1016/j.ejc.2005.07.019}
  {\path{doi:10.1016/j.ejc.2005.07.019}}.

\bibitem{Berstel&Reutenauer:2011}
Jean Berstel and Christophe Reutenauer.
\newblock {\em Noncommutative Rational Series With Applications}, volume 137 of
  {\em Encyclopedia of Mathematics and Its Applications}.
\newblock Cambridge Univ. Press, 2011.

\bibitem{BSCRF-2014}
Francine Blanchet{-}Sadri, James~D. Currie, Narad Rampersad, and Nathan Fox.
\newblock Abelian complexity of fixed point of morphism 0 {\(\mapsto\)} 012, 1
  {\(\mapsto\)} 02, 2 {\(\mapsto\)} 1.
\newblock {\em Integers}, 14:A11, 2014.
\newblock URL:
  \url{http://math.colgate.edu/\%7Eintegers/o11/o11.Abstract.html}.

\bibitem{Brown&Freedman:1987}
Thomas~C. Brown and Allen~R. Freedman.
\newblock Arithmetic progressions in lacunary sets.
\newblock {\em Rocky Mountain J. Math.}, 17:587--596, 1987.

\bibitem{Brown-2012}
Tom Brown.
\newblock Approximations of additive squares in infinite words.
\newblock {\em Integers}, 12(5):805--809, a22, 2012.
\newblock \href {https://doi.org/10.1515/integers-2012-0006}
  {\path{doi:10.1515/integers-2012-0006}}.

\bibitem{CarpiMaggi2001}
Arturo Carpi and Cristiano Maggi.
\newblock On synchronized sequences and their separators.
\newblock {\em Theor. Inform. Appl.}, 35(6):513--524, 2001.
\newblock \href {https://doi.org/10.1051/ita:2001129}
  {\path{doi:10.1051/ita:2001129}}.

\bibitem{CCSS2014}
Julien Cassaigne, James~D. Currie, Luke Schaeffer, and Jeffrey Shallit.
\newblock Avoiding three consecutive blocks of the same size and same sum.
\newblock {\em J. ACM}, 61(2):Art. 10, 17, 2014.
\newblock \href {https://doi.org/10.1145/2590775} {\path{doi:10.1145/2590775}}.

\bibitem{CFZ-2000}
Julien Cassaigne, S{\'e}bastien Ferenczi, and Luca~Q. Zamboni.
\newblock Imbalances in {Arnoux}-{Rauzy} sequences.
\newblock {\em Ann. Inst. Fourier}, 50(4):1265--1276, 2000.
\newblock \href {https://doi.org/10.5802/aif.1792}
  {\path{doi:10.5802/aif.1792}}.

\bibitem{CRSZ-2011}
Julien Cassaigne, Gw{\'e}na{\"e}l Richomme, Kalle Saari, and Luca~Q. Zamboni.
\newblock Avoiding abelian powers in binary words with bounded abelian
  complexity.
\newblock {\em Int. J. Found. Comput. Sci.}, 22(4):905--920, 2011.
\newblock \href {https://doi.org/10.1142/S0129054111008489}
  {\path{doi:10.1142/S0129054111008489}}.

\bibitem{CWW-2019}
Jin Chen, Zhixiong Wen, and Wen Wu.
\newblock On the additive complexity of a {T}hue-{M}orse-like sequence.
\newblock {\em Discrete Appl. Math.}, 260:98--108, 2019.
\newblock \href {https://doi.org/10.1016/j.dam.2019.01.008}
  {\path{doi:10.1016/j.dam.2019.01.008}}.

\bibitem{Cobham72}
Alan Cobham.
\newblock Uniform tag sequences.
\newblock {\em Math. Systems Theory}, 6:164--192, 1972.
\newblock \href {https://doi.org/10.1007/BF01706087}
  {\path{doi:10.1007/BF01706087}}.

\bibitem{Coven-Hedlund-1973}
Ethan~M. Coven and G.~A. Hedlund.
\newblock Sequences with minimal block growth.
\newblock {\em Math. Systems Theory}, 7:138--153, 1973.
\newblock \href {https://doi.org/10.1007/BF01762232}
  {\path{doi:10.1007/BF01762232}}.

\bibitem{Currie-Rampersad-2011}
James Currie and Narad Rampersad.
\newblock Recurrent words with constant abelian complexity.
\newblock {\em Adv. Appl. Math.}, 47(1):116--124, 2011.
\newblock \href {https://doi.org/10.1016/j.aam.2010.05.001}
  {\path{doi:10.1016/j.aam.2010.05.001}}.

\bibitem{Durand2011}
Fabien Durand.
\newblock Cobham's theorem for substitutions.
\newblock {\em Journal of the European Mathematical Society}, 13(6):1799--1814,
  September 2011.
\newblock URL: \url{https://ems.press/doi/10.4171/jems/294}, \href
  {https://doi.org/10.4171/jems/294} {\path{doi:10.4171/jems/294}}.

\bibitem{Fici-etal-2016a}
Gabriele Fici, Alessio Langiu, Thierry Lecroq, Arnaud Lefebvre, Filippo
  Mignosi, Jarkko Peltom\"aki, and \'Elise Prieur-Gaston.
\newblock Abelian powers and repetitions in {S}turmian words.
\newblock {\em Theoret. Comput. Sci.}, 635:16--34, 2016.
\newblock \href {https://doi.org/10.1016/j.tcs.2016.04.039}
  {\path{doi:10.1016/j.tcs.2016.04.039}}.

\bibitem{Fici-etal-2016b}
Gabriele Fici, Alessio Langiu, Thierry Lecroq, Arnaud Lefebvre, Filippo
  Mignosi, and \'Elise Prieur-Gaston.
\newblock Abelian repetitions in {S}turmian words.
\newblock In {\em Developments in {L}anguage {T}heory}, volume 7907 of {\em
  Lecture Notes in Comput. Sci.}, pages 227--238. Springer, Heidelberg, 2013.
\newblock \href {https://doi.org/10.1007/978-3-642-38771-5\_21}
  {\path{doi:10.1007/978-3-642-38771-5\_21}}.

\bibitem{Halbeisen&Hungerbuhler:2000}
Lorenz Halbeisen and Norbert {Hungerb\"uhler}.
\newblock An application of {Van der Waerden's} theorem in additive number
  theory.
\newblock {\em INTEGERS}, 0:\#A7, 2000.
\newblock Available online at
  \url{https://math.colgate.edu/~integers/a7/a7.pdf}.

\bibitem{KK-2017}
Idrissa Kabor\'{e} and Boucar\'{e} Kient\'{e}ga.
\newblock Abelian complexity of {T}hue-{M}orse word over a ternary alphabet.
\newblock In {\em Combinatorics on words}, volume 10432 of {\em Lecture Notes
  in Comput. Sci.}, pages 132--143. Springer, Cham, 2017.
\newblock URL: \url{https://doi.org/10.1007/978-3-319-66396-8_13}, \href
  {https://doi.org/10.1007/978-3-319-66396-8\_13}
  {\path{doi:10.1007/978-3-319-66396-8\_13}}.

\bibitem{Loth97}
M.~Lothaire.
\newblock {\em Combinatorics on words}.
\newblock Cambridge Mathematical Library. Cambridge University Press,
  Cambridge, 1997.
\newblock \href {https://doi.org/10.1017/CBO9780511566097}
  {\path{doi:10.1017/CBO9780511566097}}.

\bibitem{MorseHedlund1940}
Marston Morse and Gustav~A. Hedlund.
\newblock Symbolic dynamics. {II}: {Sturmian} trajectories.
\newblock {\em Am. J. Math.}, 62:1--42, 1940.
\newblock \href {https://doi.org/10.2307/2371431} {\path{doi:10.2307/2371431}}.

\bibitem{Walnut1}
Hamoon Mousavi.
\newblock Automatic theorem proving in {W}alnut, 2016.
\newblock Preprint available at \url{https://arxiv.org/abs/1603.06017}.

\bibitem{ParreauRigoRowlandVandomme-2015}
Aline Parreau, Michel Rigo, Eric Rowland, and {\'E}lise Vandomme.
\newblock A new approach to the 2-regularity of the {{\(\ell\)}}-abelian
  complexity of 2-automatic sequences.
\newblock {\em Electron. J. Comb.}, 22(1):research paper p1.27, 44, 2015.
\newblock URL:
  \url{www.combinatorics.org/ojs/index.php/eljc/article/view/v22i1p27}.

\bibitem{Pirillo&Varricchio:1994}
Giuseppe Pirillo and Stefano Varricchio.
\newblock On uniformly repetitive semigroups.
\newblock {\em Semigroup Forum}, 49:125--129, 1994.

\bibitem{Rao2015}
Micha\"{e}l Rao.
\newblock On some generalizations of abelian power avoidability.
\newblock {\em Theoret. Comput. Sci.}, 601:39--46, 2015.
\newblock \href {https://doi.org/10.1016/j.tcs.2015.07.026}
  {\path{doi:10.1016/j.tcs.2015.07.026}}.

\bibitem{RSZ-2010}
Gw\'{e}na\"{e}l Richomme, Kalle Saari, and Luca~Q. Zamboni.
\newblock Balance and abelian complexity of the {T}ribonacci word.
\newblock {\em Adv. in Appl. Math.}, 45(2):212--231, 2010.
\newblock \href {https://doi.org/10.1016/j.aam.2010.01.006}
  {\path{doi:10.1016/j.aam.2010.01.006}}.

\bibitem{RSZ-2011}
Gw\'{e}na\"{e}l Richomme, Kalle Saari, and Luca~Q. Zamboni.
\newblock Abelian complexity of minimal subshifts.
\newblock {\em J. Lond. Math. Soc. (2)}, 83(1):79--95, 2011.
\newblock \href {https://doi.org/10.1112/jlms/jdq063}
  {\path{doi:10.1112/jlms/jdq063}}.

\bibitem{RigoMaes2002}
Michel Rigo and Arnaud Maes.
\newblock More on generalized automatic sequences.
\newblock {\em Journal of Automata, Languages, and Combinatorics},
  7(3):351--376, 2002.
\newblock \href {https://doi.org/10.25596/jalc-2002-351}
  {\path{doi:10.25596/jalc-2002-351}}.

\bibitem{RSW-2023}
Michel Rigo, Manon Stipulanti, and Markus~A. Whiteland.
\newblock Automaticity and {P}arikh-collinear morphisms.
\newblock In {\em Combinatorics on words}, volume 13899 of {\em Lecture Notes
  in Comput. Sci.}, pages 247--260. Springer, Cham, 2023.
\newblock \href {https://doi.org/10.1007/978-3-031-33180-0\_19}
  {\path{doi:10.1007/978-3-031-33180-0\_19}}.

\bibitem{RSW-2024}
Michel Rigo, Manon Stipulanti, and Markus~A. Whiteland.
\newblock Automatic abelian complexities of {P}arikh-collinear fixed points,
  2024.
\newblock To be published in Theory Comput. Syst. Preprint available at
  \url{https://arxiv.org/abs/2405.18032}.

\bibitem{Sahasrabudhe-2015}
Julian Sahasrabudhe.
\newblock Sturmian words and constant additive complexity.
\newblock {\em Integers}, 15:Paper No. A30, 8, 2015.

\bibitem{Shallit1988}
Jeffrey Shallit.
\newblock A generalization of automatic sequences.
\newblock {\em Theoret. Comput. Sci.}, 61(1):1--16, 1988.
\newblock \href {https://doi.org/10.1016/0304-3975(88)90103-X}
  {\path{doi:10.1016/0304-3975(88)90103-X}}.

\bibitem{Shallit-2021}
Jeffrey Shallit.
\newblock Abelian complexity and synchronization.
\newblock {\em Integers}, 21:Paper No. A36, 14, 2021.

\bibitem{Walnut2}
Jeffrey Shallit.
\newblock {\em The logical approach to automatic sequences---exploring
  combinatorics on words with {\tt {W}alnut}}, volume 482 of {\em London
  Mathematical Society Lecture Note Series}.
\newblock Cambridge University Press, Cambridge, 2023.

\bibitem{Shallit-2023}
Jeffrey Shallit.
\newblock Note on a {F}ibonacci parity sequence.
\newblock {\em Cryptogr. Commun.}, 15(2):309--315, 2023.
\newblock \href {https://doi.org/10.1007/s12095-022-00592-5}
  {\path{doi:10.1007/s12095-022-00592-5}}.

\bibitem{Sloane}
Neil J.~A. Sloane and et~al.
\newblock The {O}n-{L}ine {E}ncyclopedia of {I}nteger {S}equences.
\newblock URL: \url{https://oeis.org}.

\bibitem{Turek-2013}
Ond\v{r}ej Turek.
\newblock Abelian complexity and abelian co-decomposition.
\newblock {\em Theoret. Comput. Sci.}, 469:77--91, 2013.
\newblock \href {https://doi.org/10.1016/j.tcs.2012.10.034}
  {\path{doi:10.1016/j.tcs.2012.10.034}}.

\bibitem{Turek-2015}
Ond\v{r}ej Turek.
\newblock Abelian complexity function of the {T}ribonacci word.
\newblock {\em J. Integer Seq.}, 18(3):Article 15.3.4, 29, 2015.

\end{thebibliography}

\end{document}